\documentclass[11pt, leqno, oneside]{amsart}

\usepackage{amsmath, amsfonts, amssymb, amsthm, mathtools}
\usepackage{enumerate}  
\usepackage{enumitem}
\usepackage{stmaryrd}
\usepackage{fullpage}
\usepackage{setspace}
\usepackage{tikz}
\usepackage{tikz-cd}
\usepackage{geometry}
\usepackage{hyperref}

\hypersetup{
	colorlinks = true,
	citecolor = violet,
  	linkcolor  = cyan,
  	urlcolor   = cyan
}

\numberwithin{equation}{section}

\newtheorem{thm}{Theorem}[section]
\newtheorem{cor}[thm]{Corollary}
\newtheorem{lem}[thm]{Lemma}
\newtheorem{prop}[thm]{Proposition}

\theoremstyle{definition}
\newtheorem{defi}[thm]{Definition}
\newtheorem{defprop}[thm]{Definition and Proposition}

\newtheorem{notation}[thm]{Notation}
\newtheorem{remark}[thm]{Remark}
\newtheorem{example}[thm]{Example}
\newtheorem{construction}[thm]{Construction}

\newcommand{\Z}{\mathbb{Z}}

\DeclareMathOperator{\rank}{rank}

\DeclareMathOperator{\cok}{cok}
\renewcommand{\phi}{\varphi}
\DeclareMathOperator{\Ima}{Im}

\DeclareMathOperator{\End}{End}

\DeclareMathOperator{\Hom}{Hom}

\DeclareMathOperator{\pd}{pd}

\DeclareMathOperator{\dep}{depth}

\DeclareMathOperator{\syz}{syz}

\newcommand{\MatFac}[3]{\textup{MF}_{#1}^{#2}(#3)}
\newcommand{\MCM}{\textup{MCM}}

\newcommand{\mc}{\mathcal}
\newcommand{\mf}[1]{\mathfrak{#1}}
\newcommand{\rvect}[1]{\begin{bmatrix} #1 \end{bmatrix}}

\newcount\colveccount
\newcommand*\colvec[1]{
        \global\colveccount#1
        \begin{pmatrix}
        \colvecnext
}
\def\colvecnext#1{
        #1
        \global\advance\colveccount-1
        \ifnum\colveccount>0
                \\
                \expandafter\colvecnext
        \else
                \end{pmatrix}
        \fi
}

\title{Matrix Factorizations with more than two factors}
\author{Tim Tribone}

\begin{document}
\maketitle

\begin{abstract}
Given an element $f$ in a regular local ring, we study matrix factorizations of $f$ with $d \ge 2$ factors, that is, we study tuples of square matrices $(\phi_1,\phi_2,\dots,\phi_d)$ such that their product is $f$ times an identity matrix of the appropriate size. Several well known properties of matrix factorizations with 2 factors extend to the case of arbitrarily many factors. For instance, we show that the stable category of matrix factorizations with $d\ge 2$ factors is naturally triangulated and we give explicit formula for the relevant suspension functor. We also extend results of Kn\"orrer and Solberg which identify the category of matrix factorizations with the full subcategory of maximal Cohen-Macaulay modules over a certain non-commutative algebra $\Gamma$. As a consequence of our findings, we observe that the ring $\Gamma$ behaves, homologically, like a ``non-commutative hypersurface ring'' in the sense that every finitely generated module over $\Gamma$ has an eventually $2$-periodic projective resolution.

\end{abstract}

\thispagestyle{empty} 

\section*{Introduction}

Let $S$ be a regular local ring and $f$ a non-zero non-unit in $S$. A matrix factorization of $f$ is a pair of $n\times n$ matrices $(\phi,\psi)$ such that $\phi\psi = f\cdot I_n$, where $I_n$ is the identity matrix of size $n$. The correspondence given by Eisenbud \cite[Corollary 6.3]{eisenbud_homological_1980} between matrix factorizations and maximal Cohen-Macaulay modules has made these objects an important tool for studying hypersurface rings. We consider a natural generalization in which the factorizations have two or more factors. In other words, we consider tuples of $n\times n$ matrices $(\phi_1,\phi_2,\dots,\phi_d)$, for some $d \ge 2$, such that $\phi_1\phi_2\cdots\phi_d = f\cdot I_n$.

Several authors, including \cite{childs_linearizing_1978,backelin_matrix_1988,herzog_linear_1991, blaser_ulrich_2017} and also \cite{bergman_can_2006,buchweitz_factoring_2007}, have studied the existence of matrix factorizations with more than 2 factors, though, both sets of authors study special cases of the definition we present. The first set of authors focus their attention mainly on matrices which have only linear entries; factorizations of this type correspond to \textit{Ulrich modules} over the hypersurface defined by $f$. The second set of authors tackle a different question. Specifically, Bergman asked in \cite{bergman_can_2006} whether the adjoint (the matrix of cofactors) of the generic $n\times n$ matrix factors into a product of two non-invertible $n\times n$ matrices.

Our approach is to study the category of matrix factorizations of $f$ with $d\ge 2$ factors. We denote this category by $\MatFac{S}{d}{f}$. In Section \ref{section:Frobenius} we examine the basic structure of $\MatFac{S}{d}{f}$. We show that $\MatFac{S}{d}{f}$ is a Frobenius category and therefore its stable category is naturally triangulated. In the case $d=2$, the stable category we present coincides with the one originally studied by Eisenbud. In particular, it is equivalent to the stable category of maximal Cohen-Macaulay $R=S/(f)$-modules. In Sections \ref{section:KRS} and \ref{section:d_branched_cover} we extend results of Kn\"orrer and Solberg which give module theoretic descriptions of $\MatFac{S}{d}{f}$ for all $d \ge 2$. Specifically, Theorem \ref{thm:R_hash_sigma_equiv_MFd} extends \cite[Proposition 2.1]{knorrer_cohen-macaulay_1987} by identifying the category of matrix factorizations with $d \ge 2$ factors with the category of maximal Cohen-Macaulay modules over the skew group algebra of the $d$-fold branched cover of $R$. The equivalence of categories given by Kn\"orrer, which we generalize, is an important initial step in the classification of simple hypersurface singularities. In Section \ref{section:OmegaMF}, we investigate the suspension functor on the stable category of matrix factorizations with $d\ge 2$ factors. As a consequence of our results, we observe that the skew group algebra mentioned above behaves like a ``non-commutative hypersurface ring'' in the sense that it is Iwanaga-Gorenstein and that any of its finitely generated modules have an eventually periodic resolution of period at most 2. Throughout, and specifically in Section \ref{section:examples}, we give examples with the goal of illustrating the differences between the categories $\MatFac{S}{2}{f}$ and $\MatFac{S}{d}{f}$ for $d > 2$.

\section{Definitions and Background}\label{section:definition_background}

In this section we collect the definitions, notation, and conventions that we will use throughout.

\begin{defi}\label{defi:MF}
Let $S$ be a regular local ring, $f$ a non-zero non-unit in $S$, and $d\ge 2$ an integer. A \textit{matrix factorization of f with $d$ factors} is a $d$-tuple of homomorphisms between finitely generated free $S$-modules of the same rank, $(\phi_1:F_2\to F_1,\phi_2:F_3\to F_2,\dots ,\phi_d:F_1\to F_d$), such that
\[\phi_1\phi_2\cdots\phi_d = f\cdot 1_{F_1}.\]
Depending on the context, we may omit the free $S$-modules in the notation and simply write $(\phi_1,\phi_2,\dots,\phi_d)$. If the free $S$-modules $F_1,\dots,F_d$ are of rank $n$, we say $(\phi_1,\phi_2,\dots,\phi_d)$ is of size $n$. 
\end{defi}

As we will see, there is an inherent cyclic nature to matrix factorizations with $d$ factors. It will therefore be convenient to adopt the following notation conventions.

\begin{notation} The letter $d$ will always be an integer indicating the number of factors in a matrix factorization. When $d$ is clear from context, all indices are taken modulo $d$ unless otherwise specified. More specifically, let $i\neq j \in \Z_d$ and let $A_1,A_2,\dots,A_d$ be symbols indexed over $\Z_d$. Let $\tilde i$ and $\tilde j$ be integer representatives of $i,j$ within the range $0 < \tilde i,\tilde j \leq d$. The notation $A_iA_{i+1}\cdots A_{j}$ will be taken to mean
\[\begin{cases}
A_iA_{i+1}\cdots A_{j-1}A_j & \text{ if } \tilde i \leq \tilde j\\
A_iA_{i+1}\cdots A_dA_1\cdots A_{j-1}A_j &\text{ if } \tilde i \ge \tilde j
\end{cases}.\]
We follow a similar convention for indexing a decreasing list of symbols over $\Z_d$. 
\end{notation}

A non-zero finitely generated module $M$ over a local ring $A$ is called \textit{maximal Cohen-Macaulay} (MCM) if $\dep_A(M) = \dim A$ (the Krull dimension of $A$). Our first observation is that matrix factorizations of $f$ with $d$ factors encode MCM modules over the \textit{hypersurface ring} $R = S/(f)$, and that Definition \ref{defi:MF} is more symmetric than it seems.

\begin{lem}\label{thm:shifts_of_X}
Let $S$ be a regular local ring and $f$ a non-zero non-unit in $S$. Let $(\phi_1:F_2\to F_1,\phi_2:F_3 \to F_2,\dots,\phi_d:F_1 \to F_d)$ be a matrix factorization of $f$ with $d\ge 2$ factors. For any $k \in \Z_d$,
\begin{enumerate}[label = (\roman*)]
    \item \label{thm:shifts_of_X:1} \(\phi_k\phi_{k+1}\cdots\phi_{k-1} = f \cdot 1_{F_k}\), and
    \item \label{thm:shifts_of_X:2} if $\cok\phi_k \neq 0$, then $\cok\phi_k$ is an MCM $R$-module.
\end{enumerate}
\end{lem}

\begin{proof} \begin{enumerate}[label = (\roman*)]
    \item We proceed by induction on $d\ge 2$. For the case $d=2$, we simply need to show that $\phi\psi = f\cdot 1_F$ implies $\psi\phi=f\cdot 1_G$. This fact is well known but we include the argument here for clarity. Suppose $(\phi:G \to F,\psi:F \to G)$ is a matrix factorization with 2 factors. Since $f$ is a non-zero element in the domain $S$, it follows that both $\phi$ and $\psi$ are injective. Canceling $\phi$ on the left of the equation $\phi\psi\phi = f \cdot \phi = \phi \cdot f$, we find $\psi\phi = f \cdot 1_G$.
    
    Now, assume $d > 2$ and that the statement holds for matrix factorizations with less than $d$ factors. Let $k \in \Z_d$ and notice that, by viewing the composition $\phi_k\phi_{k+1}: F_{k+2}\to F_k$ as a single homomorphism, the $(d-1)$-tuple \[(\phi_1,\phi_2,\dots,\phi_{k-1},\phi_k\phi_{k+1},\phi_{k+2},\dots,\phi_d)\] is a matrix factorization of $f$ with $d-1$ factors. By induction, it follows that $\phi_k\phi_{k+1}\cdots\phi_{k-1} = f \cdot 1_{F_{k}}$.
    \item Let $k \in \Z_d$. By \ref{thm:shifts_of_X:1}, we have that $\phi_k\phi_{k+1}\cdots\phi_{k-1}=f \cdot 1_{F_k}$. In particular, $f \cdot \cok\phi_k = 0$, that is, $\cok\phi_k$ is an $R$-module. Also, as in \ref{thm:shifts_of_X:1}, the homomorphism $\phi_k$ is injective since $f \in S$ is non-zero. Thus, we have a short exact sequence
    \[\begin{tikzcd}
    0 \rar &F_{k+1} \rar{\phi_k} &F_k \rar &\cok\phi_k \rar &0,
    \end{tikzcd}\] which implies that $\pd_S(\cok\phi_k) \leq 1$. By the Auslander-Buchsbaum formula, we have that
    \[\dep(\cok\phi_k) = \dim(S) - \pd_S(\cok\phi_k) \ge \dim(S) - 1 = \dim(R).\] That is, $\cok\phi_k$ is an MCM $R$-module.
\end{enumerate}
\end{proof}

\begin{defi} Let $S$ be a regular local ring and $f$ a non-zero non-unit in $S$. Let $X = (\phi_1:F_2\to F_1,\dots,\phi_d: F_1\to F_d)$ and $X' = (\phi_1':F_2'\to F_1',\dots,\phi_d': F_1'\to F_d')$ be matrix factorizations of $f$.
\begin{enumerate} [label = (\roman*)]
\item  A \textit{morphism of matrix factorizations} between $X$ and $X'$ is a $d$-tuple of $S$-module homomorphisms, $\alpha = (\alpha_1,\alpha_2,\dots,\alpha_d)$, making each square of the following diagram commute:
\[\begin{tikzcd}
F_1 \dar{\alpha_1} \rar{\phi_d} &F_d \dar{\alpha_d} \rar{\phi_{d-1}} &\cdots \rar{\phi_2}  &F_2 \dar{\alpha_2}\rar{\phi_1} &F_1 \dar{\alpha_1}\\
F_1' \rar{\phi_d'} &F_d' \rar{\phi_{d-1}'} &\cdots \rar{\phi_2'} &F_2' \rar{\phi_1'} &F_1'.
\end{tikzcd}\]
Composition of morphisms is defined component-wise, that is, if $\alpha = (\alpha_1,\dots,\alpha_d): X \to X''$ and $\beta =(\beta_1,\dots,\beta_d): X' \to X$ are morphisms of matrix factorizations, then $\alpha\circ \beta= (\alpha_1\beta_1,\alpha_2\beta_2,\dots,\alpha_d\beta_d): X' \to X''$. The matrix factorizations $X$ and $X'$ are isomorphic if there exists a morphism $\alpha = (\alpha_1,\dots,\alpha_d):X \to X'$ such that $\alpha_k$ is an isomorphism for each $k\in \Z_d$.

\item Let $\MatFac{S}{d}{f}$ denote the category of matrix factorizations of $f$ with $d$ factors. The additive structure on $\MatFac{S}{d}{f}$ is given by the direct sum of $X$ and $X'$
\[\scalebox{.97}{\(X \oplus X' \coloneqq \left(\phi_1\oplus \phi_1':F_2\oplus F_2' \to F_1\oplus F_1',\dots,\phi_d\oplus\phi_d': F_1\oplus F_1' \to F_d\oplus F_d'\right).\)}\]

\item Motivated by Lemma \ref{thm:shifts_of_X}, we define functors $T^j:\MatFac{S}{d}{f} \to \MatFac{S}{d}{f}$, $j \in \Z_d$, by
\[T^j(\phi_1,\phi_2,\dots,\phi_d) = (\phi_{j+1}, \phi_{j+2},\dots,\phi_{j-1},\phi_j)\] and
\[T^j(\alpha_1,\alpha_2,\dots,\alpha_d) = (\alpha_{j+1},\alpha_{j+2},\dots,\alpha_{j-1},\alpha_j)\] for any $(\alpha_1,\alpha_2,\dots,\alpha_d) \in \Hom_{\MatFac{S}{d}{f}}(X,X')$. We refer to $T=T^1$ as the \textit{shift functor} on $\MatFac{S}{d}{f}$.
\end{enumerate}
\end{defi}

\begin{defi}
Let $S$ be a regular local ring with maximal ideal $\mf n$ and let $f$ be a non-zero non-unit in $S$. Set $R=S/(f)$ and fix $d\ge 2$. 
\begin{enumerate}[label = (\roman*)]
    \item Let $\MCM(R)$ denote the category of \text{maximal Cohen-Macaulay} $R$-modules.
    \item An $R$-module $M$ is \textit{stable} if it has no direct summands isomorphic to $R$.
    \item For an $R$-module $M$, let $\syz_R^1(M)$ denote the \textit{reduced} first syzygy of $M$.
    \item A matrix factorization $X=(\phi_1,\phi_2,\dots,\phi_d)$ is called \textit{stable} if $\cok\phi_k$ is a stable $R$-module for all $k \in \Z_d$.
    \item A homomorphism between free $S$-modules $\phi: G \to F$ is called \textit{minimal} if $\Ima\phi \subseteq \mf n F$. 
    \item A matrix factorization $(\phi_1,\phi_2,\dots,\phi_d) \in\MatFac{S}{d}{f}$ is called \textit{reduced} if $\phi_k$ is minimal for all $k \in \Z_d$.
    \item A non-zero matrix factorization $X \in \MatFac{S}{d}{f}$ is \textit{indecomposable} if $X \cong X' \oplus X''$ implies $X' = 0$ or $X'' = 0$.
\end{enumerate}
\end{defi}

Eisenbud showed in \cite{eisenbud_homological_1980} that there is a one-to-one correspondence between reduced matrix factorizations (with 2 factors) and stable MCM $R$-modules. This correspondence can be realized in the form of a stable equivalence between the categories $\MatFac{S}{2}{f}$ and $\MCM(R)$. We record two consequences that will be needed later.

\begin{prop}\label{thm:Eisenbud_d=2}\cite{eisenbud_homological_1980}
Let $S$ be a regular local ring, $f$ a non-zero non-unit in $S$, and set $R= S/(f)$.
\begin{enumerate}[label = (\roman*)]
    \item \label{thm:Eisenbud_d=2:existence} For any MCM $R$-module $M$, there exists a matrix factorization $(\phi,\psi)\in \MatFac{S}{2}{f}$ such that $\phi$ is minimal and $\cok\phi \cong M$.
    \item \label{thm:Eisenbud_d=2:reduced} A matrix factorization $(\phi,\psi) \in \MatFac{S}{2}{f}$ is reduced if and only if it is stable. In this case, $\syz_R^1(\cok\phi) \cong \cok\psi$ and $\syz_R^1(\cok\psi) \cong \cok\phi$.
\end{enumerate}
\end{prop}

We will see in Section \ref{section:OmegaMF} that only one direction of \ref{thm:Eisenbud_d=2:reduced} holds when $d>2$. Finally, we state another observation, also made by Eisenbud, that will help us identify matrix factorizations.

\begin{lem}\label{thm:n_equals_m}\cite[Corollary 5.4]{eisenbud_homological_1980}
Let $S$ be a regular local ring and $f$ a non-zero non-unit in $S$. Suppose $A: G \to F$ and $B:F \to G$ are homomorphisms of finitely generated free $S$-modules such that $AB = f \cdot 1_F$. Then, $\rank F = \rank G$.
\end{lem}

\begin{proof}
Once again, since $f$ is a non-zerodivisor, the maps $A$ and $B$ are injective. Since $f\cdot \cok A = 0$, $\cok A$ is a torsion $S$-module. From the short exact sequence
\[\begin{tikzcd}
0 \rar &G\rar{A} &F \rar &\cok A \rar &0
\end{tikzcd}\] it follows that $\rank_S F = \rank_S G$.
\end{proof}

\section{$\MatFac{S}{d}{f}$ is a Frobenius Category}\label{section:Frobenius}

In this section we show that there is a natural choice of an exact structure on the category $\MatFac{S}{d}{f}$ which induces the structure of a triangulated category on the stable category $\underline{\textup{MF}}_S^d(f)$ defined below. We start by recalling the axioms that define an exact category. The axioms and definitions below follow the presentation given in \cite{buhler_exact_2010} and we refer the reader to this paper for more information on exact categories. 

Let $\mc A$ be an additive category. A pair of composable morphisms $\begin{tikzcd} A' \rar{i} &A \rar{p} &A'' \end{tikzcd}$ is called a \textit{kernel-cokernel pair} if $i$ is a kernel of $p$ and $p$ is a cokernel of $i$. Given a collection of kernel-cokernel pairs, $\mc E$, we call a morphism $i:A' \to A$ an \textit{admissible monomorphism} if there exists a morphism $p: A \to A''$ such that $\begin{tikzcd} A' \rar{i} &A \rar{p} &A'' \end{tikzcd}$ is an element of $\mc E$. Dually, a morphism $p: A\to A''$ is an \textit{admissible epimorphism} is there exists a morphism $i:A' \to A$ such that their composition is in $\mc E$. We will indicate admissible monomorphisms and admissible epimorphisms by the arrows $\rightarrowtail$ and $\twoheadrightarrow$ respectively.

An \textit{exact structure} on $\mc A$ is a class $\mc E$ of kernel-cokernel pairs which is closed under isomorphisms and such that the following axioms hold:
\begin{enumerate}[itemsep = 5pt]
    \item[(E0)]\label{E0} The identity morphism $1_X$ is an admissible monomorphism for all $X \in \mc A$.
    \item[(E0$^{\textup{op}}$)]\label{E0op} The identity morphism $1_X$ is an admissible epimorphism for all $X \in \mc A$.
    \item[(E1)]\label{E1} Admissible monomorphisms are closed under composition.
    \item[(E1$^{\textup{op}}$)]\label{E1op} Admissible epimorphisms are closed under composition.
    \item[(E2)]\label{E2} The push-out of an admissible monomorphism $X \rightarrowtail Y$ and an arbitrary morphism $X \to X'$ exists and induces an admissible monomorphism as in the diagram
    \[\begin{tikzcd}
    X \dar \rar[tail] &Y\dar[dotted]\\
    X' \rar[dotted,tail] &Y'.
    \end{tikzcd}\]
    \item[(E2$^{\textup{op}}$)]\label{E2op} The pull-back of an admissible epimorphism $X'' \twoheadrightarrow Y'$ and an arbitrary morphism $Y \to Y'$  exists and induces an admissible epimorphism as in the diagram
    \[\begin{tikzcd}
    X \dar[dotted] \rar[two heads,dotted] &Y \dar\\
    X' \rar[two heads] &Y'.
    \end{tikzcd}\]
\end{enumerate}
Given an additive category $\mc A$ and a class $\mc E$ satisfying these axioms, the pair $(\mc A,\mc E)$ is called an \textit{exact category}.

Let $S$ be a regular local ring, $f \in S$ a non-zero non-unit, and set $R=S/(f)$. Fix an integer $d\ge 2$ and let $X=(\phi_1:F_2 \to F_1,\dots,\phi_d:F_1\to F_d), X'=(\phi_1':F_2'\to F_1',\dots,\phi_d':F_1'\to F_d'),$ and $X'' = (\phi_1'': F_2'' \to F_1'',\dots,\phi_d'':F_1'' \to F_d'')$ be matrix factorizations in $\MatFac{S}{d}{f}$.

\begin{defi}
Suppose we have a pair of morphisms $\alpha = (\alpha_1,\dots,\alpha_d): X \to X''$ and $\beta=(\beta_1,\dots,\beta_d):X' \to X$ in $\MatFac{S}{d}{f}$. Then, the composition
\[\begin{tikzcd}
&X' \rar{\beta} &X \rar{\alpha} &X''
\end{tikzcd} \] is called a \textit{short exact sequence of matrix factorizations} if the sequence 
\[\begin{tikzcd}
&0 \rar &F_k' \rar{\beta_k} &F_k \rar{\alpha_k} &F_k'' \rar &0
\end{tikzcd}\] is a short exact sequence of free $S$-modules for each $k \in \Z_d$.
\end{defi}

\begin{lem}
A short exact sequence of matrix factorizations is a kernel-cokernel pair in $\MatFac{S}{d}{f}$.
\end{lem}

\begin{proof}
Let \begin{tikzcd}
X' \rar{\beta} &X \rar{\alpha} &X''
\end{tikzcd} be a short exact sequence of matrix factorizations. First we show that $\beta$ is the kernel of $\alpha$. Certainly, we have that $\alpha\beta = 0$. Suppose $g: Y \to X$ is another morphism such that $\alpha g = 0$, where $Y = (\psi_1:G_2\to G_1,\dots,\psi_d: G_1 \to G_d)$. Let $k\in\Z_d$. We have the following diagram of free $S$-modules
\[\begin{tikzcd}
&&G_k \dlar[dotted,swap]{\tilde g_k}\dar{g_k} \drar{0}&\\
0\rar&F_k' \rar{\beta_k} &F_k \rar{\alpha_k} &F_k''\rar &0
\end{tikzcd}\] where the bottom row is exact. Since $\beta_k$ is the kernel of $\alpha_k$, there exists a unique $S$-homomorphism $\tilde g_k: G_k \to F_k'$ such that $\beta_k \tilde g_k = g_k$. It suffices to show that $\tilde g =(\tilde g_1,\tilde g_2,\dots, \tilde g_d): Y \to X'$ is a morphism of matrix factorizations since each $\tilde g_k$, $k \in\Z_d$, is uniquely determined. That is, we need to show that the diagram
\[\begin{tikzcd}
G_{k+1} \dar{\tilde g_{k+1}} \rar{\psi_k} &G_k \dar{\tilde g_k}\\
F_{k+1}' \rar{\phi_k'} &F_k'
\end{tikzcd}\] commutes for all $k \in \Z_d$. Note that $g_k\psi_k = \phi_kg_{k+1}$ and $\phi_k\beta_{k+1}=\beta_k\phi_k'$ since $g$ and $\beta$ are morphisms in $\MatFac{S}{d}{f}$. Then, $\beta_k\tilde g_k \psi_k = g_k\psi_k = \phi_kg_{k+1}= \phi_k \beta_{k+1}\tilde g_{k+1} = \beta_k \phi_k' \tilde g_{k+1}$ and since $\beta_k$ is injective, we can cancel it on the left to conclude that $\tilde g_k \psi_k = \phi_k' \tilde g_{k+1}$ as desired. Hence, $\tilde g$ is the unique morphism such that 
\[\begin{split}
    \beta \circ \tilde g &=(\beta_1\tilde g_1,\dots, \beta_d\tilde g_d)\\
    &= (g_1,\dots, g_d)\\
    &= g.
\end{split}\] This completes the proof that $\beta$ is a kernel of $\alpha$. The proof that $\alpha$ is a cokernel of $\beta$ is similar.
\end{proof}

Let $\mathcal E_d$ denote the class of short exact sequences of matrix factorizations in $\MatFac{S}{d}{f}$. The first four axioms of an exact category are satisfied by the pair $(\MatFac{S}{d}{f},\mc E_d)$ directly from the definitions. The axioms $[\textup{E}2]$ and $[\textup{E}2^{\textup{op}}]$ also hold, which we will show below. Before we do, we need to know more about the form of the admissible morphisms in $(\MatFac{S}{d}{f}, \mc E_d)$.

\begin{lem}\label{thm:admiss_morphs_form}
Let $\gamma = (\gamma_1,\dots,\gamma_d): X \to X''$ be a morphism of matrix factorizations.
\begin{enumerate}
    \item $\gamma$ is an admissible epimorphism if and only if the $S$-homomorphisms $\gamma_1,\dots,\gamma_d$ are surjections.
    \item $\gamma$ is an admissible monomorphism if and only if the $S$-homomorphisms $\gamma_1,\dots,\gamma_d$ are split injections.
\end{enumerate}
\end{lem}

\begin{proof} We prove only (2) as the proof of (1) is similar. Suppose $\gamma$ is an admissible monomorphism. Then, there exists an admissible epimorphism $\pi=(\pi_1,\pi_2,\dots,\pi_d): X \to X''$ such that $\begin{tikzcd}
X' \rar[tail]{\gamma} &X \rar[two heads]{\pi} &X''
\end{tikzcd}$ is a short exact sequence of matrix factorizations. In particular,
\[\begin{tikzcd}
0 \rar &F_k' \rar{\gamma_k} &F_k \rar{\pi_k} &F_k'' \rar &0
\end{tikzcd}\] is a short exact sequence of $S$-modules for each $k\in \Z_d$. Since $F_k''$ is free, this sequence is split and therefore $\gamma_k$ is a split injection.

Conversely, suppose the homomorphisms $\gamma_1,\dots,\gamma_d$ are each split injections. For $k\in \Z_d$, set $F_k'' \coloneqq  \cok\gamma_k$ and $\pi_k: F_k \to F_k''$ the natural projection map. Notice that $F_k''$ is a free $S$-module of rank equal to $\rank F_k - \rank F_k'$. Now, for each $k \in \Z_d$, there exists a map $\phi_k'':F_{k+1}\to F_k$ such that the following diagram with split exact rows commutes:
\[\begin{tikzcd}
&0 \rar &F_1' \dar{\phi_d'} \rar{\gamma_1} &F_1 \rar{\pi_1} \dar{\phi_d} &F_1'' \dar[dotted]{\phi_d''}\rar &0\\
&0 \rar &F_d' \rar{\gamma_d} \dar{\phi_{d-1}'} &F_d \rar{\pi_d}\dar{\phi_{d-1}} &F_d'' \rar \dar[dotted]{\phi_{d-1}''} &0\\
&&\vdots \dar{\phi_2'} &\vdots \dar{\phi_2} &\vdots \dar[dotted]{\phi_2''} &\\
&0 \rar &F_2' \rar{\gamma_2} \dar{\phi_1'} &F_2 \rar{\pi_2}\dar{\phi_1} &F_2'' \dar[dotted]{\phi_1''}\rar &0\\
&0 \rar &F_1' \rar{\gamma_1} &F_1 \rar{\pi_1} &F_1'' \rar &0.
\end{tikzcd}\] In particular, there exists $t_1:F_1'' \to F_1$ such that $\pi_1t_1 = 1_{F_1''}$. The splitting allows us to compute the composition along the right most column:
\[\begin{split}
    \phi_1''\phi_2'' \cdots \phi_d'' &= \pi_1\phi_1\phi_2\cdots\phi_d t_1\\
                                    &= f\cdot \pi_1 t_1\\
                                    &= f \cdot 1_{F_1''}.
\end{split}\] Since the free modules $F_1'',F_2'',\dots,F_d''$ are all of the same rank, we have that $X'' = (\phi_1'':F_2'' \to F_1'',\dots,\phi_d'': F_1'' \to F_d'') \in \MatFac{S}{d}{f}$ and
\[\begin{tikzcd}[column sep = 3.7em]
X' \rar[tail]{\gamma} &X \rar[two heads]{(\pi_1,\dots,\pi_d)} &X''
\end{tikzcd}\] is a short exact sequence of matrix factorizations.

\end{proof}

Lemma \ref{thm:admiss_morphs_form} indicates that not every monomorphism of matrix factorizations is an admissible monomorphism. The simplest example of this arises when $d=2$. 

\begin{example}Suppose $(\phi:G \to F,\psi: F \to G) \in \MatFac{S}{2}{f}$ with $\cok\psi\neq 0$. Then, the tuple $(\psi,1_G)$ forms a morphism between the matrix factorizations $(\phi,\psi) \to (f\cdot 1_G , 1_G)$. This morphism is a monomorphism, in the sense that it can be cancelled on the left, but it is not admissible since the cokernel of $\psi$ is not a free $S$-module. 

The same is true of epimorphisms, that is, there are epimorphisms that are not admissible. For $(\phi,\psi)\in \MatFac{S}{2}{f}$ with $\cok\psi\neq 0$, the tuple $(1_F,\psi)$ forms a morphism between the matrix factorizations $(f\cdot 1_F,1_F) \to (\phi,\psi)$. If $(a,b)\circ(1_F,\psi) = (a',b')\circ(1_F,\psi)$ for some morphisms $(a,b),(a',b')$, then $a=a'$ and $b\psi = b'\psi$. We can pre-compose both sides of the second equation with $\phi$ to get $b\cdot f = b' \cdot f$, hence $b=b'$. So, $(1_F,\psi)$ can be cancelled on the right but is not admissible epimorphism since $\psi$ is not surjective. 

Actually, further inspection of these examples shows they are both monomorphisms and both epimorphisms but neither is admissible of either type. In particular, neither is an isomorphism. In Abelian category, a monomorphism which is also an epimorphism must be an isomorphism. Similar examples can be constructed for all $d>2$ and therefore we note that the category $\MatFac{S}{d}{f}$ is not Abelian for any $d\ge 2$.
\end{example}

\begin{prop}
The collection $\mc E_d$ of short exact sequences of matrix factorizations in $\MatFac{S}{d}{f}$ satisfies the axioms $[\textup{E}2]$ and $[\textup{E}2^{\textup{op}}]$.
\end{prop}

\begin{proof}
We will show that $[\textup{E}2]$ holds. The proof that $[\textup{E}2^{\textup{op}}]$ is satisfied is similar. Suppose we have a diagram in $\MatFac{S}{d}{f}$
\begin{equation}\label{diagram:pushout}
\begin{tikzcd}
X \dar{\beta}\rar[tail]{q} &Y\\
X'
\end{tikzcd}\end{equation} where $Y = (\psi_1:G_2\to G_1,\dots,\psi_d: G_1 \to G_d)$, $\beta=(\beta_1,\dots,\beta_d)$, and $q = (q_1,\dots,q_d)$. Let $k \in \Z_d$. We may take the push-out of the injection $q_k$ and the map $\beta_k$ which yields the diagram
\[\begin{tikzcd}
0 \rar &F_k \dar{\beta_k} \rar{q_k} &G_k \dar{p_k} \rar &\cok q_k \dar[equals] \rar &0\\
0 \rar &F_k' \rar{\alpha_k} &P_k \rar &\cok q_k \rar &0.
\end{tikzcd}\] We make the following observations from this diagram: Since the morphism $q$ is an admissible monomorphism, the map $q_k$ is a split injection. Hence, $\cok q_k$ is a free $S$-module. It follows that the bottom sequence also splits and so $P_k$ is free with $\rank_S P_k = \rank_S F_k' + \rank_S G_k - \rank_S F_k$. This also implies that $\alpha_k$ is a split injection. Since $k$ was arbitrary, this yields $d$ free $S$-modules $P_1,P_2,\dots,P_d$, each of the same rank, and $d$-tuples $\alpha =(\alpha_1,\dots,\alpha_d)$ and $p = (p_1,\dots,p_d)$.

Next, let $k \in \Z_d$ and consider the diagram
\[\begin{tikzcd}
F_k \dar{\beta_k} \rar{q_k} &G_k \dar{p_k} \arrow[ddr, bend left, "p_{k-1}\psi_{k-1}"] &\\
F_k' \ar[drr, bend right,swap, "\alpha_{k-1}\phi_{k-1}'"] \rar{\alpha_k} &P_k \drar[dotted]{\chi_{k-1}} &\\
&&P_{k-1}.
\end{tikzcd}\]
There is a unique homomorphism $\chi_{k-1}: P_k \to P_{k-1}$ depicted above since $$p_{k-1}\psi_{k-1}q_k = p_{k-1}q_{k-1}\phi_{k-1} = \alpha_{k-1}\beta_{k-1}\phi_{k-1}= \alpha_{k-1}\phi_{k-1}'\beta_k.$$ In particular, the map $\chi_{k-1}$ is given by \[\chi_{k-1}(\overline{(a_{k'},b_k)}) = \alpha_{k-1}\phi_{k-1}'(a_k) + p_{k-1}\psi_{k-1}(b_k) = \overline{(\phi_{k-1}'(a_k),\psi_{k-1}(b_k))} \in P_{k-1},\] for any $(a_k,b_k) \in F_k'\oplus G_k$. In other words, $\chi_{k-1}$ is the map induced by the direct sum $\phi_{k-1}'\oplus \psi_{k-1}$ on the quotients $P_k \to P_{k-1}$. These maps link together to form a sequence of compositions
\[\begin{tikzcd}
P_1 \rar{\chi_d} &P_d \rar{\chi_{d-1}} &\cdots \rar{\chi_2} &P_2 \rar{\chi_1} &P_1.
\end{tikzcd}\] From the explicit description of $\chi_k$ we have that $\chi_1\chi_2\cdots\chi_d = f\cdot 1_{P_1}$. Since the free $S$-modules $P_1,\dots,P_d$ are all of the same rank, it follows that $Y' = (\chi_1:P_2\to P_1,\cdots, \chi_d: P_1 \to P_d)$ is a matrix factorization of $f$. 

It is not hard to see that $\alpha: X' \to Y'$ and $p:Y\to Y'$ form morphisms of matrix factorizations and that these morphisms render \eqref{diagram:pushout} a commutative square. As we showed above, the map $\alpha_k$ is a split injection for all $k \in \Z_d$. Hence, $\alpha$ is an admissible monomorphism by Lemma \ref{thm:admiss_morphs_form}. To finish the proof, it suffices to check the necessary universal property which we omit as it is a straightforward computation.

\end{proof}

\begin{cor}
The pair $(\MatFac{S}{d}{f},\mc E_d)$ is an exact category. \qed
\end{cor}

With the exact structure on $\MatFac{S}{d}{f}$ fixed, we will often omit reference to $\mc E_d$. We proceed now with the main goal of this section: to show that the exact category $\MatFac{S}{d}{f}$ is a Frobenius category. First, we recall the necessary definitions which can also be found in \cite{buhler_exact_2010}.

An object in an exact category $(\mc A, \mc E)$ is called \textit{projective}, respectively \textit{injective}, if it satisfies the usual lifting property with respect to admissible epimorphisms, respectively admissible monomorphisms. The pair $(\mc A,\mc E)$ is said to have \textit{enough projectives} if for every object $X \in \mc A$, there is an admissible epimorphism $P \twoheadrightarrow X$ with $P$ projective. Dually, $(\mc A, \mc E)$ has \textit{enough injectives} if for every object $X \in \mc A$, there is an admissible monomorphism $X \rightarrowtail I$ with $I$ injective. The exact category $(\mc A, \mc E)$ is said to be a \textit{Frobenius category} if it has enough projectives, enough injectives, and the classes of projective objects and injective objects coincide.

\begin{defi}
For each $i \in \Z_d$, let $\mc P_i$ denote the matrix factorization of size $1$ whose $i$th component is multiplication by $f$ on $S$ while the rest are the identity on $S$. In other words, $\mc P_i$ is given by the composition
\[\begin{tikzcd}
S_1 \rar{1} &S_d \rar{1} &\cdots \rar{1} &S_{i+1} \rar{f} &S_i \rar{1} &\cdots \rar{1} &S_2 \rar{1} &S_1
\end{tikzcd}\] where $S_k = S$ for each $k \in \Z_d$. We also set $\mc P = \bigoplus_{i \in \Z_d} \mc P_i$.
\end{defi}

\begin{lem}\label{thm:shift_of_projs}
Let $X \in \MatFac{S}{d}{f}$ and $j \in \Z_d$. Then, $X$ is projective (respectively injective) if and only if $T^j(X)$ is projective (respectively injective).
\end{lem}

\begin{proof}
Suppose $X$ is projective and $j \in \Z_d$. Let $\alpha=(\alpha_1,\dots,\alpha_d): X' \to X''$ be an admissible morphism and let $p=(p_1,\dots,p_d): T^j(X) \to X''$ be any morphism. Then, we have morphisms $T^{-j}(\alpha): T^{-j}(X') \to T^{-j}(X'')$ and $T^{-j}(p): X \to T^{-j}(X'')$. The characterization of admissible epimorphisms in Lemma \ref{thm:admiss_morphs_form} implies that $T^{-j}(\alpha)$ is also an admissible epimorphism. Since $X$ is projective, there exists $q = (q_1,q_2,\dots,q_d) : X \to T^{-j}(X')$ such that $T^{-j}(\alpha)q = T^{-j}(p)$. Applying $T^j$ we find that $\alpha T^j(q) = p$ implying that $T^j(X)$ is projective. The proof of the converse and both directions regarding injectivity are similar.
\end{proof}

\begin{lem}\label{thm:P_i_proj_inj} The matrix factorizations $\mc P_1,\mc P_2,\dots,\mc P_d$ are both projective and injective in $\MatFac{S}{d}{f}$.
\end{lem}

\begin{proof}
Directly from the definition we see that $T^j(\mc P_i) = \mc P_{i-j}$ for any $i,j \in \Z_d$. Therefore, because of Lemma \ref{thm:shift_of_projs}, it suffices to show that $\mc P_1$ is both projective and injective. We shall start by showing that $\mc P_1$ is projective.

Suppose $\alpha=(\alpha_1,\dots,\alpha_d): X \to X''$ is an admissible epimorphism and $p=(p_1,\dots,p_d): \mc P_1 \to X''$ is an arbitrary morphism. We would like to complete the diagram 
\begin{equation}\label{diagram:proj_P_1}\begin{tikzcd}
&\mc P_1\dar{p} \dlar[dotted,swap]{q}\\
X \rar[two heads]{\alpha}&X''
\end{tikzcd}\end{equation} One component of this diagram is the following diagram of free $S$-modules
\[\begin{tikzcd}
&S \dlar[dotted,swap]{q_1} \dar{p_1}\\
F_1 \rar{\alpha_1}&F_1'' \rar &0.
\end{tikzcd}\]
Since $S$ is free and $\alpha_1$ is surjective, there exists a map $q_1: S \to F_1$ such that $\alpha_1q_1 = p_1$. We can use this map to construct a morphism of matrix factorizations $\mc P_1 \to X$ which makes \eqref{diagram:proj_P_1} commute. Let $q = (q_1,\phi_2\phi_3\cdots\phi_d q_1, \phi_3\cdots\phi_d q_1,\dots, \phi_d q_1)$. The fact that $q$ forms a morphism $\mc P_1 \to X$ can be seen in the following diagram of $S$-modules
\[\begin{tikzcd}[column sep = large]
S \rar{1} \dar{q_1} &S \rar{1} \dar{\phi_d q_1} &\cdots \rar{1} &S \rar{1} \dar{\phi_3\cdots\phi_d q_1} &S \rar{f} \dar{\phi_2\phi_3\cdots\phi_d q_1} &S \dar{q_1}\\
F_1 \rar{\phi_d} &F_d \rar{\phi_{d-1}} &\cdots \rar{\phi_3} &F_3 \rar{\phi_2} &F_2 \rar{\phi_1} &F_1.
\end{tikzcd}\]

Finally, 
\[\begin{split}
    \alpha q &= (\alpha_1q_1, \alpha_2\phi_2\cdots\phi_dq_1,\dots,\alpha_d\phi_dq_1)\\
             &= (p_1, \phi_2''\cdots\phi_d''\alpha_1q_1, \cdots, \phi_d''\alpha_1q_1)\\
             &= (p_1, \phi_2''\cdots \phi_d''p_1,\cdots,\phi_d''p_1)\\
             &= (p_1,p_2,\dots,p_d)\\
             &=p
\end{split}\] which implies that $\mc P_1$ is projective.

In order to show that $\mc P_1$ is an injective matrix factorization, let $\beta = (\beta_1,\dots,\beta_d): X' \rightarrowtail X$ be an admissible monomorphism and $a=(a_1,\dots,a_d): X' \to \mc P_1$ be an arbitrary morphism. We would like to complete the diagram
\[\begin{tikzcd}
&X' \dar{a}\rar[tail]{\beta} & X \dlar[dotted]{b}\\
&\mc P_1&
\end{tikzcd}\]Since $\beta$ is an admissible monomorphism, each component $\beta_k$ is split. In particular, there exists a map $t: F_2' \to F_2$ such that $t\beta_2 = 1_{F_2}$. This splitting allows us to build the morphism \[b = (a_2t\phi_2'\cdots\phi_d',a_2t,a_2t\phi_2',\dots,a_2t\phi_2'\cdots\phi_{d-1}'): X \to \mc P_1\] and this morphism satisfies $b\beta = a$. 
\end{proof}

In Lemma \ref{thm:only_indecomp_projs}, we will see that $\mc P_1,\mc P_2,\dots, \mc P_d$ are the only indecomposable projective (respectively injective) matrix factorizations up to isomorphism.

The next step is to show that $\MatFac{S}{d}{f}$ has enough projectives and enough injectives. Along the way, we construct the syzygy and cosyzygy of a matrix factorization and give explicit formulas for each.

\begin{remark}
Let $X \in \MatFac{S}{d}{f}$. By tensoring with $R = S/(f)$, which we denote here by $\overline{\square} = \square \otimes_S R$, one can associate to $X$ an infinite chain of free $R$-modules:
\[\begin{tikzcd}
\cdots \rar{\overline \phi_{2}} &\overline F_2 \rar{\overline\phi_1} &\overline F_1 \rar{\overline\phi_d} &\overline F_d \rar{\overline\phi_{d-1}} &\cdots \rar{\overline \phi_2} &\overline F_2 \rar{\overline \phi_1} &\overline F_1 \rar{\overline \phi_d} &\cdots
\end{tikzcd}\] In the case $d=2$, this chain is a acyclic and, by truncating appropriately, it forms a free resolution over $R$ of $\cok\phi_1$ (or of $\cok\phi_2$). However, if $d>2$, this chain is not acyclic. In fact, it is not even a complex. Instead, it is precisely an \textit{acyclic} $d$-\textit{complex} (see \cite{iyama_derived_2017}). With this perspective in mind, it is likely that the formulas given below can be obtained as lifted versions of the ones found in \cite[Section 2]{iyama_derived_2017}. We give explicit constructions without referencing any such lifting.
\end{remark}

\begin{construction}\label{construction:enough}
Let $X = (\phi_1:F_2\to F_1,\dots,\phi_d:F_1\to F_d) \in \MatFac{S}{d}{f}$ be of size $n$. We construct short exact sequences \[\begin{tikzcd}
K \rar[tail] &P \rar[two heads] &X
\end{tikzcd} \text{ and } \begin{tikzcd}
X \rar[tail] &I \rar[two heads] &K'
\end{tikzcd}\] with $P$ projective and $I$ injective and give explicit formula for the resulting syzygy and the cosyzygy of $X$.

To start, set $\widehat F_k = \bigoplus_{i=1}^{d-1}F_{k+i}$. For each $k \in \Z_d$, define $S$-homomorphisms $D_k: F_{k+1}\oplus \widehat F_{k+1} \to F_k \oplus \widehat F_k$ and $D_k': F_{k+1}\oplus \widehat F_{k+1} \to F_k \oplus \widehat F_k$ by
\[D_k(a_{k+1},a_{k+2},\dots,a_{k-1},a_k) = (fa_k,a_{k+1},\dots,a_{k-1})\] and
\[D_k'(a_{k+1},a_{k+2},\dots,a_{k-1},a_k) = (a_k,fa_{k+1},a_{k+2},\dots,a_{k-1})\]
for all $a_i \in F_i, i \in \Z_d$. Set $P(X) = (D_1,D_2,\dots,D_d)$ and $I(X) = (D_1',D_2',\dots,D_d')$. Then, the $d$-tuples $P(X)$ and $I(X)$ form matrix factorizations of $f$ both isomorphic to $\bigoplus_{i=1}^{d} \mc P_i^n$.

For each $i,k \in \Z_d$, define a homomorphism $\theta_{ki}^X: F_i \to F_k$ given by
\[\theta_{ki}^X =\begin{cases}
1_{F_k} & i=k\\
\phi_k\phi_{k+1}\cdots\phi_{i-2}\phi_{i-1} & i\neq k.\\
\end{cases}\]
Then, for each $k \in \Z_d$, define $\Theta_k^X: \widehat F_k \to F_k$ and $\Xi_k^X: F_k \to \widehat F_k$ by 
\[\Theta_k^X(a_{k+1},a_{k+2},\dots,a_{k-1}) = \sum_{i\neq k}\theta_{ki}^X(a_i)\]
and 
\[\Xi_k^X(a_k) = \left(\theta_{(k+1)k}^X(a_k),\theta_{(k+2)k}^X(a_k),\dots,\theta_{(k-1)k}^X(a_k)\right).\]
Let $k \in \Z_d$ and consider the following diagram:
\begin{equation}\label{diagram:building_enough_projs}\begin{tikzcd}
0 \rar &\widehat F_{k+1} \dar[dotted]{\Omega_k}\rar{\epsilon_{k+1}^X} &F_{k+1}\oplus \widehat F_{k+1} \dar{D_k} \rar{\rho_{k+1}^X} &F_{k+1} \dar{\phi_k}\rar &0\\
0 \rar &\widehat F_{k} \rar{\epsilon_{k}^X} &F_{k}\oplus \widehat F_{k} \rar{\rho_{k}^X} &F_{k} \rar &0
\end{tikzcd}\end{equation} where $\rho_k^X = \begin{pmatrix}1_{F_k} &\Theta_k^X \end{pmatrix}$ and $\epsilon_k^X = \begin{pmatrix} -\Theta_k^X\\ 1_{\widehat F_k}\end{pmatrix}$. The rows are split exact sequences of free $S$-modules and one can check that right most square commutes by recalling that $\phi_k\theta_{(k+1)i} = \theta_{ki}$ for $i\neq k$ and $\phi_k\theta_{(k+1)k} = f\cdot 1_{F_k}$. Thus, there is an induced map $\Omega_k: \widehat F_{k+1} \to \widehat F_k$ as depicted. Since the rows are split, $\Omega_k$ can be computed by using the splitting, that is, $\Omega_k = \pi_kD_k\epsilon_{k+1}$ where $\pi_k: F_k \oplus \widehat F_k \to \widehat F_k$ is projection onto $\widehat F_k$. In particular, 
\[\begin{split}
    \Omega_k(a_{k+2},a_{k+3},\dots,a_{k-1},a_k) &= \pi_kD_k\epsilon_{k+1}(a_{k+2},a_{k+3},\dots,a_{k-1},a_k)\\
    &= \pi_kD_k\left(-\sum_{i\neq k+1}\theta_{(k+1)i}(a_i),a_{k+2},\dots,a_{k-1},a_k\right)\\
    &=\pi_k\left(fa_k,-\sum_{i\neq k+1}\theta_{(k+1)i}(a_i),a_{k+2},\dots,a_{k-1}\right)\\
    &=\left(-\sum_{i\neq k+1}\theta_{(k+1)i}(a_i),a_{k+2},\dots,a_{k-1}\right)
\end{split}\]and therefore we can represent the components of $\Omega_k$ as 
\[\Omega_{k} =\begin{pmatrix}
-\theta_{(k+1)(k+2)} &  -\theta_{(k+1)(k+3)} & -\theta_{(k+1)(k+4)}   & \ldots &-\theta_{(k+1)(k-1)} & -\theta_{(k+1)(k)}\\
1_{F_{k+2}} &  0 & 0 & \ldots & 0& 0\\
0 & 1_{F_{k+3}} & 0  &\ldots &0 &0\\
\vdots & \vdots &1_{F_{k+4}} &\ddots  &\vdots & \vdots \\
0  &   0 &0 &\ddots &0 & 0\\
0 & 0 &0 &\cdots &1_{F_{k-1}}& 0\\
\end{pmatrix}.\]
Since $k$ was arbitrary, we have a $d$-tuple $(\Omega_1,\Omega_2,\dots,\Omega_d)$ which has the property that
\[\Omega_1\Omega_2\cdots\Omega_d = \pi_1D_1D_2\cdots D_d \epsilon_1 = f\pi_1\epsilon_1 = f \cdot 1_{\widehat F_1}.\]
Since the free $S$-modules $\widehat F_1,\widehat F_2,\dots,\widehat F_d$ are all of the same rank, it follows that $(\Omega_1,\dots,\Omega_d) \in \MatFac{S}{d}{f}$. We denote this matrix factorization by $\Omega_{\MatFac{S}{d}{f}}(X)$ and refer to it as the \textit{syzygy} of $X$. Combining the diagrams \eqref{diagram:building_enough_projs} for all $k \in \Z_d$ we have a short exact sequence
\begin{equation}\label{equation:enough_proj}\begin{tikzcd}\Omega_{\MatFac{S}{d}{f}}(X) \rar[tail]{\epsilon} &P(X) \rar[two heads]{\rho} &X\end{tikzcd}\end{equation} where $\rho = (\rho_1,\rho_2,\dots,\rho_d)$ and $\epsilon = (\epsilon_1,\epsilon_2,\dots,\epsilon_d)$.

Similarly, we have a matrix factorization $\Omega^{-}_{\MatFac{S}{d}{f}}(X) = (\Omega_1^-,\Omega_2^-,\dots,\Omega_d^-)$, the \textit{cosyzygy} of $X$, and a short exact sequence of matrix factorizations
induced by the commutative diagrams 
\begin{equation}\label{diagram:building_enough_injs}\begin{tikzcd}
0 \rar &F_{k+1} \dar{\phi_k}\rar{\lambda_{k+1}^X} &F_{k+1}\oplus \widehat F_{k+1} \dar{D_k'}\rar{\eta_{k+1}^X} &\widehat F_{k+1} \dar[dotted]{\Omega_k^-}\rar &0\\
0 \rar &F_{k} \rar{\lambda_{k}^X} &F_{k}\oplus \widehat F_{k} \rar{\eta_{k}^X} &\widehat F_{k} \rar &0,
\end{tikzcd}\end{equation} for all $k \in \Z_d$, where $\eta_k^X = \begin{pmatrix} -\Xi_k^X & 1_{\widehat F_k}\end{pmatrix}$, $\lambda_k^X = \begin{pmatrix} 1_{F_k}\vspace{2mm}\\ \Xi_k^X\end{pmatrix}$. The induced short exact sequence is
\begin{equation}\label{equation:enough_inj}\begin{tikzcd}
X \rar[tail]{\lambda} &I(X) \rar[two heads]{\eta} &\Omega_{\MatFac{S}{d}{f}}^-(X),
\end{tikzcd}\end{equation} where $\eta = (\eta_1,\eta_2,\dots,\eta_d)$ and $\lambda=(\lambda_1,\lambda_2,\dots,\lambda_d)$.
Finally, we can represent the components of $\Omega_k^-$ by
\[\Omega_k^{-}=\begin{pmatrix}
0           &0           &\cdots &0           &0           &-\theta_{(k+1)k}\\
1_{F_{k+2}} &0           &\cdots &0           &0           &-\theta_{(k+2)k}\\
0           &1_{F_{k+3}} &\ddots &0           &0           &-\theta_{(k+3)k}\\
\vdots      &\vdots      &\ddots &\vdots      &\vdots           &\vdots\\
0           &0           &\cdots &1_{F_{k-2}} &0           &-\theta_{(k-2)k}\\
0           &0           &\cdots &0           &1_{F_{k-1}} &-\theta_{(k-1)k}
\end{pmatrix}.\]
\end{construction}

\begin{example}\label{example:d=2_3}
Let $(\phi,\psi) \in \MatFac{S}{2}{f}$. Then,
\[\Omega_{\MatFac{S}{2}{f}}(\phi,\psi) = (-\psi,-\phi) \cong (\psi,\phi)\]
and
\[\Omega^{-}_{\MatFac{S}{2}{f}}(\phi,\psi) = (-\psi,-\phi) \cong (\psi,\phi).\]
In particular, both $\Omega_{\MatFac{S}{2}{f}}(-)$ and $\Omega^{-}_{\MatFac{S}{2}{f}}(-)$ are isomorphic to the shift functor when $d=2$. In this case, $(\phi,\psi)$ is reduced if and only if $\Omega_{\MatFac{S}{2}{f}}(\phi,\psi)$ (respectively $\Omega_{\MatFac{S}{2}{f}}^-(\phi,\psi)$) is reduced. However, for $X \in \MatFac{S}{d}{f}$ with $d>2$, neither $\Omega_{\MatFac{S}{2}{f}}(X)$ nor $\Omega_{\MatFac{S}{2}{f}}^-(X)$ will be reduced. For instance, if $X= (\phi_1:F_2\to F_1,\phi_2: F_3\to F_2,\phi_3: F_1\to F_3) \in \MatFac{S}{3}{f}$ is of size $n$, then
\[\Omega_{\MatFac{S}{3}{f}}(X) = 
\left(\begin{pmatrix}
-\phi_2 & -\phi_2\phi_3\\
1_{F_3} &0
\end{pmatrix},\begin{pmatrix}
-\phi_3 & -\phi_3\phi_1\\
1_{F_1} &0
\end{pmatrix},\begin{pmatrix}
-\phi_1 & -\phi_1\phi_2\\
1_{F_2} &0
\end{pmatrix}\right)\] and
\[\Omega^{-}_{\MatFac{S}{3}{f}}(X) = 
\left(\begin{pmatrix}
0 & -\phi_2\phi_3\\
1_{F_3} &-\phi_3
\end{pmatrix},\begin{pmatrix}
0 & -\phi_3\phi_1\\
1_{F_1} &-\phi_1
\end{pmatrix},\begin{pmatrix}
0 & -\phi_1\phi_2\\
1_{F_2} &-\phi_2
\end{pmatrix}\right)\]
which are of size $2n$.

Another observation that can be made from these formulas is that, for each $k \in \Z_3$, we have an isomorphism of $R$-modules \[\cok \begin{pmatrix}
-\phi_{k+1} & -\phi_{k+1}\phi_{k+2}\\
1_{F_{k+2}} &0\\
\end{pmatrix} \cong \cok(\phi_{k+1}\phi_{k+2}) \cong \syz_R^1(\cok\phi_k) \oplus R^{m_k}\] for some $m_k \ge 0$. A similar statement is true for $\Omega_{\MatFac{S}{3}{f}}^-(X)$ and more generally we have the following proposition.
\end{example}

\begin{prop}\label{thm:cok_Omega_k}
Let $X \in \MatFac{S}{d}{f}$ be of size $n$. Let $\Omega_{\MatFac{S}{d}{f}}(X)$ and $\Omega^{-}_{\MatFac{S}{d}{f}}(X)$ be the matrix factorizations constructed in \ref{construction:enough}. Then, for each $k \in \Z_d$,
\[\cok(\Omega_k) \cong \syz_R^1(\cok\phi_k) \oplus R^{m_k} \cong \cok(\Omega_k^{-})\] where $m_k = n - \mu_R(\cok\phi_k)$.
\end{prop}

\begin{proof}
Let $k \in \Z_d$. The diagram \eqref{diagram:building_enough_projs} induces the diagram
\[\begin{tikzcd}
&&0\dar&0\dar&0\dar\\
&0 \rar &\widehat F_{k+1} \dar \rar{\Omega_k} &\widehat F_k \dar \rar &\cok\Omega_k \dar \rar &0\\
&0 \rar  &F_{k+1}\oplus \widehat F_{k+1} \dar \rar{D_k} &F_k \oplus \widehat F_k \dar \rar &F_k/fF_k \dar \rar &0\\
&0 \rar &F_{k+1} \dar \rar{\phi_k} &F_k \dar \rar &\cok\phi_k \dar\rar &0\\
&&0&0&0
\end{tikzcd}\] with exact rows and columns. The right most column displays $\cok\Omega_k$ as a (not necessarily reduced) syzygy of $\cok\phi_k$ over $R$ as desired. 

Alternatively, we can see from the explicit formulas for $\Omega_k$ and $\Omega_k^-$ that 
\[\cok\Omega_k \cong \cok\theta_{(k+1)k}^X \cong \cok\Omega_k^-.\] Now, recall that $\theta_{(k+1)k} = \phi_{k+1}\phi_{k+2}\cdots\phi_{k-1}$. Since $(\phi_k,\theta_{(k+1)k})$ is a matrix factorization with 2 factors, we have that $\cok(\theta_{(k+1)k})\cong \syz_R^1(\cok\phi_k)\oplus R^{m_k}$ for some $m_k \ge 0$. In particular, $m_k = n - \mu_R(\cok\phi_k)$ by the uniqueness of minimal free resolutions over $R$.
\end{proof}

Returning to our original goal, we note that since the matrix factorizations $P(X) \cong I(X) \cong \bigoplus_{i=1}^d \mc P_i^n$ are projective and injective by Lemma \ref{thm:P_i_proj_inj}, the sequences \eqref{equation:enough_proj} and \eqref{equation:enough_inj} imply that $\MatFac{S}{d}{f}$ has enough projectives and enough injectives. Additionally, we have the following.

\begin{prop}\label{thm:proj_iff_inj}
An object $X \in \MatFac{S}{d}{f}$ is projective if and only if it is injective. 
\end{prop}

\begin{proof}
The matrix factorizations $P(X)$ and $I(X)$ are both projective and injective by Lemma \ref{thm:P_i_proj_inj}. Let $X \in \MatFac{S}{d}{f}$. If $X$ is projective, \eqref{equation:enough_proj} implies that $X$ is a summand of the injective matrix factorization $P(X)$ and therefore is injective. Conversely, if $X$ is injective, \eqref{equation:enough_inj} implies it is a summand of the projective $I(X)$.
\end{proof}

We have therefore established our original goal of this section.

\begin{thm}\label{thm:MFd_Frob}
The category $\MatFac{S}{d}{f}$ is a Frobenius category. \qed
\end{thm}

For matrix factorizations $X,X' \in \MatFac{S}{d}{f}$, let $I(X,X')$ denote the set of morphisms $X$ to $X'$ that factor through an injective matrix factorization. The stable category $\underline{\textup{MF}}_S^d(f)$ is formed by taking the same objects as $\MatFac{S}{d}{f}$ and morphisms given by the quotient \[\Hom_{\underline{\textup{MF}}_S^d(f)}(X,X') =\Hom_{\MatFac{S}{d}{f}}(X,X')/I(X,X').\] A consequence of Theorem \ref{thm:MFd_Frob} is that $\underline{\textup{MF}}_S^d(f)$ carries the structure of a triangulated category with suspension functor given by $\Omega_{\MatFac{S}{d}{f}}^{-}(-)$ \cite{happel_triangulated_1988}. We call a morphism $(\alpha_1,\alpha_2,\dots,\alpha_d): X \to X'$ \textit{null-homotopic} if there exist $S$-homomorphisms $s_j : F_j \to F_{j-1}'$, $j \in \Z_d$, such that
\[\alpha_i = \sum_{k \in \Z_d}\theta_{i(i-k)}^{X'}s_{i-k+1}\theta_{(i-k+1)i}^X\] for each $i \in \Z_d$. We denote by $\textup{HMF}_S^d(f)$ the \textit{homotopy category of matrix factorizations} which has the same objects as $\MatFac{S}{d}{f}$ and, for any $X,X' \in \MatFac{S}{d}{f}$, has morphisms
\[\Hom_{\textup{HMF}_S^d(f)}(X,X') = \Hom_{\MatFac{S}{d}{f}}(X,X')/\sim,\] where $\sim$ is the equivalence relation $\alpha \sim \alpha'$ if and only if $\alpha-\alpha'$ is null-homotopic.

\begin{prop}\label{thm:stable_homotopy}
The stable category $\underline{\textup{MF}}_S^d(f)$ and the homotopy category $\textup{HMF}_S^d(f)$ coincide.
\end{prop}

\begin{proof}
Let $X,X' \in \MatFac{S}{d}{f}$. It suffices to show that a morphism $\alpha: X \to X'$ is null-homotopic if and only if it factors through the morphism $\lambda^X : X \to I(X)$. The proof relies on an explicit description of morphisms $I(X) \to X'$.
Indeed, if $\beta: I(X) \to X'$ is any morphism, then, for any $k,j \in \Z_d$ with $j\neq 1$, we have a commutative diagram
\[\begin{tikzcd}[font = \small]
F_{k+j-1}\oplus \widehat F_{k+j-1} \dar{\beta_{k+j-1}}\rar{D_{k+j-2}'} & F_{k+j-2}\oplus \widehat F_{k+j-2} \dar{\beta_{k+j-2}}\rar{D_{k+j-3}'} & \cdots \rar{D_{k+1}'} &F_{k+1}\oplus \widehat F_{k+1} \dar{\beta_{k+1}}\rar{D_k'} &F_{k}\oplus \widehat F_{k} \dar{\beta_{k}}\\
F_{k+j-1}' \rar{\phi_{k+j-2}'} &F_{k+j-2}' \rar{\phi_{k+j-3}'} &\cdots \rar{\phi_{k+1}'} &F_{k+1}' \rar{\phi_k'} &F_{k}'.
\end{tikzcd}\]
Write the components of $\beta_i$ as \[\beta_i = \begin{bmatrix}
\beta_{ii} &\beta_{i(i+1)} &\beta_{i(i+2)} &\cdots &\beta_{i(i-1)}
\end{bmatrix}\] for some $\beta_{i(i+s)}: F_{i+s}\to F_i'$. Notice that the maps $D_k',D_{k+1}',\dots,D_{k+j-2}'$ are the identity on $F_{k+j}$. Hence,
\[\beta_{k(k+j)} = \phi_k'\phi_{k+1}'\cdots\phi_{k+j-2}'\beta_{(k+j-1)(k+j)} = \theta^{X'}_{k(k+j-1)}\beta_{(k+j-1)(k+j)}\] by the commutativity of the outer most rectangle of the diagram above. Therefore,
\[\beta_k = \begin{bmatrix}\theta_{k(k-1)}^{X'}\beta_{(k-1)k} &\beta_{k(k+1)} &\theta_{k(k+1)}^{X'}\beta_{(k+1)(k+2)} &\cdots &\theta_{k(k-2)}^{X'}\beta_{(k-2)(k-1)}\end{bmatrix}.\]

Now, if $\beta: I(X) \to X'$ is such that $\beta\lambda^X = \alpha$, then 
\[\alpha_k = \beta_k\lambda_k^X = -\sum_{i \in \Z_d}\theta_{k(k-i)}^{X'}\beta_{(k-i)(k-i+1)}\theta_{(k-i+1)k}^X\] for any $k \in \Z_d$. This says precisely that $\alpha$ is null-homotopic via the maps $\{-\beta_{(j-1)j}\}_{j\in\Z_d}$.

Conversely, if $\alpha$ is null-homotopic via maps $s_j: F_j \to F_{j-1}'$, then it is not hard to see that the maps
\[\gamma_k = \begin{bmatrix}\theta_{k(k-1)}^{X'}s_k &s_{k+1} &\theta_{k(k+1)}^{X'}s_{k+2} &\cdots &\theta_{k(k-1)}^{X'}s_{k-1} \end{bmatrix},\] for $k \in \Z_d$, form a morphism $\gamma=(\gamma_1,\gamma_2,\dots,\gamma_d): I(X) \to X'$ such that $\alpha = \gamma\lambda^X$.
\end{proof}

To end this section, we give explicit formula for the mapping cone of a morphism. Suppose $\alpha: X \to X'$ is a morphism in $\MatFac{S}{d}{f}$. The \textit{mapping cone} of $\alpha$ is the matrix factorization  \(C(\alpha) = (\Delta_1,\Delta_2, \dots, \Delta_d)\) where
\[\Delta_k = \begin{pmatrix}
\phi_k' &0              &0           &\cdots &0         &\alpha_k\\
0       &0              &0           &\cdots &0         &-\theta_{(k+1)k}^X\\
0       &1_{F_{k+2}}    &0           &\cdots &0         &-\theta_{(k+2)k}^X\\
0       &0              &1_{F_{k+3}} &\ddots &0         &-\theta_{(k+3)k}^X\\
\vdots  &\vdots         &            &\ddots &\vdots    &\vdots\\
0       &0              &\cdots      &\cdots &1_{F_{k-1}}         &-\theta_{(k-1)k}^X\\
\end{pmatrix}: F_{k+1}' \oplus \widehat F_{k+1} \to F_k' \oplus \widehat F_k\] for all $k \in \Z_d$. The cone of $\alpha$ fits into a commutative diagram
\[\begin{tikzcd}
X \dar{\alpha} \rar[tail]{\lambda^X} &I(X) \dar{\beta} \rar[two heads]{\eta^X} &\Omega_{\MatFac{S}{d}{f}}^{-}(X) \dar[equals]\\
X' \rar[tail]{q} &C(\alpha) \rar[two heads]{p} &\Omega_{\MatFac{S}{d}{f}}^{-}(X)
\end{tikzcd}\] where $p_k = \begin{pmatrix} 0 & 1_{\widehat F_k}
\end{pmatrix}$, $q_k = \begin{pmatrix} 1_{F_k'} \\ 0 \end{pmatrix}$, and $\beta_k = \begin{pmatrix}\alpha_k & 0\\ -\Xi_k^X & 1_{\widehat F_k}\end{pmatrix}$ for all $k \in \Z_d$, and $p=(p_1,\dots,p_d), q=(q_1,\dots,q_d),$ and $\beta=(\beta_1,\dots,\beta_d)$.

\section{$\MatFac{S}{d}{f}$ as a category of modules}\label{section:KRS}
Let $(S,\mf n, \bf k)$ be a regular local ring, $f$ a non-zero non-unit in $S$, and $d\ge 2$ an integer. In this section, we will show that $\MatFac{S}{d}{f}$ is equivalent to the category of MCM modules over a certain non-commutative $S$-algebra which is finitely generated and free as an $S$-module. This extends a result of Solberg \cite[Proposition 1.3]{solberg_hypersurface_1989} for all $d \ge 2$. As a consequence we will conclude that, if $S$ is complete, the Krull-Remak-Schmidt Theorem (KRS) holds in $\MatFac{S}{d}{f}$.

Recall from Section \ref{section:Frobenius} the projective (equivalently injective) matrix factorizations $\mc P_1,\mc P_2,\dots,\mc P_d$, and their direct sum $\mc P=\bigoplus_{i \in \Z_d}\mc P_i.$ Our first step is to understand the homomorphisms from $\mc P_j$ to $\mc P_i$ for any $i,j \in \Z_d$. 

\begin{defi}
For $i\neq j \in \Z_d$, let $e_{ij}\in \Hom_S(S,S)^d$ denote the $d$-tuple of homomorphisms such that the $j+1,j+2,\dots,i-1,i$ components are multiplication by $f$ while the rest are the identity on $S$. For each $i \in \Z_d$, let $e_{ii} = 1_{\mc P_i}$, the identity on $\mc P_i$.
\end{defi}

For instance, if $i \in \Z_d$, the $i$th map in $e_{i(i-1)}$ is multiplication by $f$ and the rest are the identity on $S$ while $e_{(i-1)i}$ is of the form $(f,f,\dots,f,1,f,\dots,f,f)$, where only the $i$th component is the identity on $S$. 

\begin{lem}\label{thm:gens_of_gamma}
For all $i,j \in \Z_d$, $\Hom_{\MatFac{S}{d}{f}}(\mc P_j,\mc P_i)= S \cdot e_{ij}$. In particular, the morphisms from $\mc P_j$ to $\mc P_i$ form a free $S$-module of rank 1.
\end{lem}
\begin{proof} The morphisms from $\mc P_j$ to $\mc P_i$ are tuples of the form $(\alpha_1,\dots,\alpha_d)$ for some $\alpha_k \in S$. If $(\alpha_1,\dots,\alpha_d) \in \Hom_{\MatFac{S}{d}{f}}(\mc P_i,\mc P_i)$, then it is easy to see that $\alpha_1=\alpha_2 = \dots = \alpha_d$ and so $\Hom_{\MatFac{S}{d}{f}}(\mc P_i,\mc P_i) = S\cdot e_{ii}$. Suppose $(\alpha_1,\dots,\alpha_d) \in \Hom_{\MatFac{S}{d}{f}}(\mc P_j,\mc P_i)$ where $i \neq j$. Consider the following diagram, which commutes since $(\alpha_1,\dots,\alpha_d)$ is a morphism of matrix factorizations:
\[\begin{tikzcd}
\mc P_{j}\arrow[d]&\cdots\arrow[r,"1"] &S\arrow[r,"1"]\arrow[d,"\alpha_{i+1}"] &S\arrow[r,"1"]\arrow[d,"\alpha_{i}"] &\cdots\arrow[r,"1"] &S\arrow[r,"f"]\arrow[d,"\alpha_{j+1}"] &S\arrow[d,"\alpha_{j}"]\arrow[r,"1"]&\cdots\\
\mc P_i&\cdots \arrow[r,"1"] &S\arrow[r,"f"] &S\arrow[r,"1"] &\cdots\arrow[r,"1"] &S\arrow[r,"1"] &S\arrow[r,"1"]&\cdots
\end{tikzcd}\]
We conclude that $\alpha_{j+1} = \alpha_{j + 2} = \cdots = \alpha_i$, $\alpha_{j} = \alpha_{j-1} = \cdots = \alpha_{i+1}$, and $\alpha_i = f\alpha_{i+1}$. Thus, each component of the morphism can be rewritten in terms of the element $\alpha_{i+1} \in S$. It follows that $(\alpha_1,\dots,\alpha_d) = \alpha_{i+1}e_{ij}$, that is, $\Hom_{\MatFac{S}{d}{f}}(\mc P_j,\mc P_i)$ is generated by $e_{ij}$ as an $S$-module. Since the components of $e_{ij}$ are given by multiplication by non-zero elements of $S$, a morphism $s\cdot e_{ij}=0$ if and only if $s=0$. Hence, $\Hom_{\MatFac{S}{d}{f}}(\mc P_j,\mc P_i)$ is in fact a free $S$-module of rank $1$ for all $i,j\in \Z_d$.
\end{proof}

Let $\Gamma = \End_{\MatFac{S}{d}{f}}\mc (\mc P)^{\textup{op}}$ where $\mc P=\bigoplus_{i=1}^d \mc P_i$. As $S$-modules, $\Gamma \cong \bigoplus_{i,j\in\Z_d}\Hom(\mc P_j,\mc P_i)$ and therefore, Lemma~\ref{thm:gens_of_gamma} implies that $\Gamma$ is a free $S$-module of rank $d^2$. For each pair $i,j \in \Z_d$, use the same symbol $e_{ij}$ to denote the image of the generator of $\Hom(\mc P_j,\mc P_i)$ under the natural inclusion $\Hom(\mc P_j,\mc P_i) \to \Gamma$. The set $\{e_{ij}\}_{i,j\in \Z_d}$ forms a basis for $\Gamma$ as an $S$-module. We record the basic rules for multiplication of the basis elements $e_{ij}$.


\begin{lem}\label{thm:properties_of_eij}
The basis elements $\{e_{ij}\}_{i,j\in\Z_d}$ satisfy the following properties.
\begin{enumerate}[itemsep = 2pt,label=(\roman*)]
    \item \label{thm:properties_of_eij:orthogonal} $e_{ij}e_{pq} \neq 0$ if and only if $i=q$ 
    \item \label{thm:properties_of_eij:idempotents} $e_{ii}^2 = e_{ii}$ for all $i \in \Z_d$
    \item \label{thm:properties_of_eij:sum_to_1} $\sum_{i=1}^d e_{ii} = 1_{\Gamma}$
    \item \label{thm:properties_of_eij:part_4} $e_{ij}e_{ii} = e_{ij}$ and $e_{jj}e_{ij} = e_{ij}$ for all $i,j \in \Z_d$
    
    \item \label{thm:properties_of_eij:left_mult_eii-1} $e_{i(i-1)} e_{ji} = \begin{cases}
fe_{(i-1)(i-1)} & \text{if } j=i-1\\
e_{j(i-1)} &\text{otherwise}
\end{cases}$ 
\item \label{thm:properties_of_eij:right_mult_ei+1i} \(e_{ij}e_{(i+1)i} = 
\begin{cases}
fe_{(i+1)(i+1)} &j=i+1\\
e_{(i+1)j} &\text{otherwise}\\
\end{cases}\)
\item \label{thm:properties_of_eij:mult_by_z} $\displaystyle\sum_{i=1}^d e_{i(i-1)}e_{jj} = e_{j(j-1)} =  e_{(j-1)(j-1)}\displaystyle\sum_{i=1}^d e_{i(i-1)}$ for all $j\in \Z_d$ 
\end{enumerate}\qed
\end{lem}

The element $\sum_{i=1}^d e_{i(i-1)} \in \Gamma$ is of particular interest because of the following.

\begin{lem}\label{thm:z^s}
Let $z=\sum_{i=1}^de_{i(i-1)}$ and $s\ge 1$ an integer. Write $s = dq + r$ for $q\ge 0$ and $0\leq r < d$. Then, 
\[z^s = f^q\sum_{i=1}^de_{i(i-r)}.\] In particular, $z^d = f\cdot 1_\Gamma$.
\end{lem}

\begin{proof}
If $s=1$ there is nothing to prove. Assume the formula holds for $s\ge 1$ and consider $z^{s+1}$. By induction,
\[z^{s+1} = z\cdot f^q\sum_{i=1}^d e_{i(i-r)}\] where $s = dq + r$, $q\ge 0$, and $0\leq r < d$. If $r= d-1$, then, by Lemma \ref{thm:properties_of_eij},
\[z \cdot \sum_{i=1}^d e_{i(i-r)} = \sum_{i=1}^d e_{(i+1)i}e_{i(i+1)} = f \cdot 1_\Gamma.\] Since $s=dq + d-1$, we have that $s+1 = d(q+1)$. Hence,
\[z^{s+1} = f^{q+1}\sum_{i=1}^de_{ii}\] as needed. If $0 \le r < d-1$, then
\[z \cdot \sum_{i=1}^d e_{i(i-r)} = \sum_{i=1}^d e_{i(i-r-1)}\] also by Lemma \ref{thm:properties_of_eij}. In this case,
\[z^{s+1} = f^{q}\sum_{i=1}^d e_{i(i-(r+1))}\] which completes the induction since $s+1 = dq + (r+1)$ with $0\leq r + 1 < d$.

\end{proof}

\begin{lem}\label{thm:e_ij=prod}
If $i,j\in \Z_d$ with $i\neq j$, then $e_{ij}$ can be written as a product of basis elements of the form $e_{\ell(\ell-1)}$. In particular,  
\(e_{ij} = e_{(j+1)j}e_{(j+2)(j+1)}\cdots e_{(i-1)(i-2)}e_{i(i-1)}.\)
\end{lem}

\begin{proof}
Let $i\neq j \in \Z_d$. Lemma \ref{thm:properties_of_eij} \ref{thm:properties_of_eij:right_mult_ei+1i} implies that, for any $\ell \neq j \in \Z_d$, the element $e_{\ell(\ell -1)}$ can be factored out of $e_{\ell j}$ on the right
\[e_{\ell j} = e_{(\ell -1)j}e_{\ell(\ell -1 )}.\]
Since $i\neq j$, we may apply this equality for $\ell = i,i-1,\dots,j+2,j+1$ which gives us the factorization
\[\begin{split}e_{ij} &= e_{(i-1)j}e_{i(i-1)}\\
&= e_{(i-2)j}e_{(i-1)(i-2)}e_{i(i-1)}\\
&= \cdots\\
&= e_{(j+1)j}e_{(j+2)(j+1)}\cdots e_{(i-1)(i-2)}e_{i(i-1)}.\end{split}\]
\end{proof}

Let $\MCM(\Gamma)$ denote the full subcategory of finitely generated left $\Gamma$-modules which are free when viewed as $S$-modules via the inclusion $S\cdot 1_\Gamma \subset \Gamma$. The rest of this section is dedicated to proving the following.

\begin{thm}\label{thm:MF_equiv_MCMgamma}
The categories $\MCM(\Gamma)$ and $\MatFac{S}{d}{f}$ are equivalent. In particular, if $S$ is complete, the Krull-Remak-Schmidt Theorem holds in $\MatFac{S}{d}{f}$.
\end{thm}

If $S$ is complete, it is known that the Krull-Remak-Schmidt Theorem holds in the category $\MCM(\Gamma)$ (for example see \cite{auslander_isolated_1986}). So, establishing the equivalence is the main concern (which does not require $S$ to be complete). We start by defining a functor $\mc H:\MCM(\Gamma)\to \MatFac{S}{d}{f}$ using the element $z = \sum_{i=1}^d e_{i(i-1)} \in \Gamma$. Let $M$ be a $\Gamma$-module in $\MCM(\Gamma)$. Lemma \ref{thm:properties_of_eij} \ref{thm:properties_of_eij:orthogonal}-\ref{thm:properties_of_eij:sum_to_1} show that $e_{11},\dots,e_{dd}$ are orthogonal idempotents such that $e_{11} + e_{22} + \dots + e_{dd} = 1_{\Gamma}$. Thus, $M$ decomposes, as an $S$-module, into
\[M = e_{11}M \oplus \cdots \oplus e_{dd}M.\] Since $M\in \MCM(\Gamma)$, each summand $e_{ii}M$ is a free $S$-module. Lemma \ref{thm:properties_of_eij} \ref{thm:properties_of_eij:mult_by_z} shows that left multiplication by $z  \in \Gamma$ defines an $S$-homomorphism between free $S$-modules $z:e_{ii}M \to e_{(i-1)(i-1)}M$ for all $i \in \Z_d$.

\begin{prop}\label{thm:def_H}
Let $M \in \MCM(\Gamma)$. The $d$-tuple of $S$-homomorphisms \[( z:e_{22}M\to e_{11}M, z:e_{33}M \to e_{22}M\dots, z: e_{11}M\to e_{dd}M)\] forms a matrix factorization of $f$ in $\MatFac{S}{d}{f}$, where each map is multiplication by $z = \sum_{i=1}^d e_{i(i-1)}\in \Gamma$.
\end{prop}
\begin{proof}
In light of Lemma~\ref{thm:z^s}, the only piece that needs justification is that each of the free $S$-modules involved are of the same rank. To see this, let $i \in \Z_d$. The composition
\[\begin{tikzcd}
e_{ii}M \rar{z^{d-1}} &e_{(i+1)(i+1)}M \rar{z} &e_{ii}M
\end{tikzcd}\] is $f$ times the identity on $e_{ii}M$. Since $e_{ii}M$ and  $e_{(i+1)(i+1)}M$ are free over $S$, Lemma \ref{thm:n_equals_m} implies that $\rank_S(e_{ii}M) = \rank_S(e_{(i+1)(i+1)}M)$.
\end{proof}

Following Proposition~\ref{thm:def_H}, the functor $\mc H$ is defined as follows: $\mc H(M) = (z:e_{22}M\to e_{11}M,\dots, z:e_{11}M\to e_{dd}M)$ for any $M \in \MCM(\Gamma)$ and, for a homomorphism $h:M \to N$ in $\MCM(\Gamma)$, define $\mc H(h) = (h\vert_{e_{11}M},\dots,h\vert_{e_{dd}M})$, where $h\vert_{e_{ii}M}$ denotes the restriction of $h$ to the $S$-direct summand $e_{ii}M$. Since $h$ is a $\Gamma$-homomorphism, $h\vert_{e_{ii}M}$ maps $e_{ii}M$ into $e_{ii}N$. Since multiplication by an element of $\Gamma$ commutes with any $\Gamma$-homomorphism, this $d$-tuple forms a morphism between the matrix factorizations $\mc H(M) \to \mc H(N)$. At this point we can prove that $\mc H$ is both full and faithful.

\begin{prop}\label{thm:H_full_faithful}
The functor $\mc H: \MCM(\Gamma) \to \MatFac{S}{d}{f}$ is full and faithful.
\end{prop}

\begin{proof}
Let $M,N \in \MCM(\Gamma)$. If $\mc H(h) = 0$ for some $\Gamma$-homomorphism $h: M \to N$, then $h\vert_{e_{ii}M} = 0$ for each $i \in \Z_d$. But this means that $h = \bigoplus_{i \in \Z_d} h\vert_{e_{ii}M} = 0$ implying that $\mc H$ is faithful. 

In order to show that $\mc H$ is full, let $(\alpha_1,\dots,\alpha_d):\mc H(M) \to \mc H(N)$ be a morphism of matrix factorizations. So, $\alpha_i: e_{ii}M \to e_{ii}N$ and we have a commutative diagram
\begin{equation}\label{diagram:H_full}\begin{tikzcd}
e_{ii}M \dar{\alpha_i} \rar{z} &e_{(i-1)(i-1)}M\dar{\alpha_{i-1}}\\
e_{ii}N \rar{z} &e_{(i-1)(i-1)}N
\end{tikzcd}\end{equation} for each $i \in \Z_d$. Let $h = \bigoplus_{j=1}^d\alpha_j : M \to N$ be the $S$-homomorphism given by $h(m) = \alpha_1(e_{11}m) + \alpha_2(e_{22}m) + \cdots + \alpha_d(e_{dd}m)$ for all $m \in M$. We claim that $h$ is in fact a $\Gamma$-homomorphism and furthermore, $\mc H(h) = (\alpha_1,\dots, \alpha_d)$. The second claim follows from the first and the definition of $\mc H$ and so our aim is to show that $h$ is a homomorphism of $\Gamma$-modules. Since $h$ is an $S$-homomorphism and $\Gamma$ is a finitely generated free $S$-module with basis $\{e_{ij}\}_{i,j\in\Z_d}$, we would be done if we showed that $e_{ij}h(m) = h(e_{ij}m)$ for all $i,j\in\Z_d$ and $m \in M$. By Lemma~\ref{thm:e_ij=prod}, it suffices to show that the elements of the form $e_{k(k-1)}$ pass through $h$ since each $e_{ij}$ is a product of elements of this form.

Let $i \in \Z_d$ and $m \in M$. By Lemma~\ref{thm:properties_of_eij}\ref{thm:properties_of_eij:mult_by_z}, multiplication by $z$ on $e_{ii}M$ (respectively on $e_{ii}N$) coincides with multiplication by the element $e_{i(i-1)}\in \Gamma$. Therefore, the diagram \eqref{diagram:H_full} implies that
\[\alpha_{i-1}(e_{i(i-1)}e_{ii}m) = e_{i(i-1)}\alpha_i(e_{ii}m).\] Since $e_{i(i-1)}e_{ii}m \in e_{(i-1)(i-1)}M$, the term on the left hand side is precisely $h(e_{i(i-1)}m)$. On the other hand, since $e_{ii}h(m) = \alpha_i(e_{ii}m)$,
\[\begin{split}
   e_{i(i-1)}h(m) &= e_{i(i-1)}e_{ii}h(m)\\
                  &= e_{i(i-1)}\alpha_i(e_{ii}m).
\end{split}\] Together, we have that $e_{i(i-1)}h(m) = h(e_{i(i-1)}m)$ as desired. Thus, $h$ is a $\Gamma$-homomorphism and $\mc H$ is full.
\end{proof}

To show that $\mc H$ is dense, we define a functor $\mc F:\MatFac{S}{d}{f} \to \MCM(\Gamma)$ which is given by $\mc F(\square) = \Hom_{\MatFac{S}{d}{f}}(\mc P,\square)$, where $\mc P=\bigoplus_{i=1}^d\mc P_i$. For a matrix factorization $X \in \MatFac{S}{d}{f}$, $\mc F(X)$ is a left $\Gamma$-module by pre-composing any morphism $\mc P\to X$ with an element of $\Gamma$. In order to show that the image of $\mc F$ does indeed land in $\MCM(\Gamma)$, we must show that $\mc F(X)$ is a free $S$-module. This requires an explicit description of the morphisms $\mc P \to X$.

Recall that $\mc P_i$ is the matrix factorization whose $i$th component is multiplication by $f$ on $S$ while the rest are the identity on $S$. For $k \in \Z_d$, let $D_k: S^d \to S^d$ be given by $D_k(a_1,\dots,a_d) = (a_1,\dots,a_{k-1},fa_k,a_{k+1},\dots,a_d)$ for any $(a_1,\dots,a_d) \in S^d$. Then, $\mc P=(D_1:S^d \to S^d,D_2:S^d\to S^d,\dots,D_d:S^d\to S^d)$.

\begin{lem}\label{thm:Hom_P_X}
Let $X=(\phi_1:F_2\to F_1,\dots,\phi_d:F_1 \to F_d)\in \MatFac{S}{d}{f}$ and let $(\alpha_1,\alpha_2,\dots,\alpha_d) \in \mc F(X) =\Hom_{\MatFac{S}{d}{f}}(\mc P,X)$. For each $k\in \Z_d$, we may write $\alpha_k =\rvect{\alpha_{k1} & \alpha_{k2} &\cdots &\alpha_{kd}}$ for some $\alpha_{ki}\in \Hom_S(S,F_k)$.
Then, for any $j\neq 0 \in \Z_d$,
\[\alpha_{k(k+j)} = \phi_k\phi_{k+1}\cdots\phi_{k+j-1}\alpha_{(k+j)(k+j)}.\]
\end{lem}

\begin{proof}
Similar to the proof of Lemma \ref{thm:stable_homotopy}, the formula follows from the following commutative diagram:
\[\begin{tikzcd}
S^d\arrow[r,"D_{k+j-1}"]\arrow[d,swap,"\alpha_{k+j}"] &S^d \dar{\alpha_{k+j-1}}\arrow[r,"D_{k+j-2}"]\dar &\cdots \arrow[r,"D_{k+1}"]& S^d \dar{\alpha_{k+1}}\arrow[r,"D_k"]\dar & S^d\arrow[d,"\alpha_k"]\\
F_{k+j}\arrow[r,"\phi_{k+j-1}"] &F_{k+j-1}\arrow[r,"\phi_{k+j-2}"]& \cdots\arrow[r,"\phi_{k+1}"] & F_{k+1}\arrow[r,"\phi_k"] &F_k.
\end{tikzcd}\] 
The commutativity of the outermost rectangle gives us that $\alpha_kD_kD_{k+1} \cdots D_{k+j-1} = \phi_k\phi_{k+1}\cdots\phi_{k+j-1}\alpha_{k+j}$. Since $j\neq 0$, the composition $D_kD_{k+1}\cdots D_{k+j-1}$ is the identity on the $(k+j)$th component of $S^d$. Therefore, if we compare the $(k+j)$th components of the homomorphisms on either side of the above equality, we find that $\alpha_{k(k+j)} = \phi_k\phi_{k+1}\cdots\phi_{k+j-1}\alpha_{(k+j)(k+j)}$ as desired.
\end{proof}

Let $X = (\phi_1:F_2 \to F_1, \cdots,\phi_d: F_1\to F_d) \in \MatFac{S}{d}{f}$ be a matrix factorization. Recall from Section \ref{section:Frobenius}, the homomorphisms $\theta_{ki}^X: F_i \to F_k$, $i,k \in \Z_d$, which are given by 
\[\theta_{ki}^X =\begin{cases}
1_{F_k} & i=k\\
\phi_k\phi_{k+1}\cdots\phi_{i-2}\phi_{i-1} & i\neq k.\\
\end{cases}\]
For each $(g_1,\dots,g_d) \in \bigoplus_{i=1}^d F_i$, we associate a $d$-tuple of $S$-homomorphisms
\[\theta^X(g_1,\dots,g_d) \coloneqq \left(\rvect{\theta_{k1}^Xg_1 & \cdots & \theta_{kd}^Xg_d}\right)_{k=1}^d.\] Here $\theta_{ki}^Xg_i$ is being identified with its image in $\Hom_S(S,F_k)$ under the natural isomorphism $F_k \cong \Hom_S(S,F_k)$. When $X$ is clear from context, we will omit the superscripts.
\begin{lem}\label{thm:hom_P_X_free}
Let $X = (\phi_1:F_2 \to F_1, \cdots,\phi_d: F_1\to F_d) \in \MatFac{S}{d}{f}$. Then, $\theta(g_1,\dots,g_d) \in \Hom_{\MatFac{S}{d}{f}}(\mc P,X)$ for any $(g_1,\dots,g_d) \in \bigoplus_{i=1}^d F_i$. Furthermore, the map $\theta: \bigoplus_{i=1}^d F_i \to \Hom_{\MatFac{S}{d}{f}}(\mc P,X)$ is an isomorphism of $S$-modules. 
\end{lem}

\begin{proof}
 First, we show that $\theta(g_1,\dots,g_d)$ as defined is in fact a morphism of matrix factorizations between $\mc P$ and $X$. What needs to be shown is the commutativity of the diagram
\[\begin{tikzcd}[ampersand replacement = \&]
\&S^d\dar[swap]{\rvect{\theta_{(k+1)1}g_1 &  \cdots   & \theta_{(k+1)d}g_d}} \arrow[r,"D_k"] \&S^d \dar{\rvect{\theta_{k1}g_1  &   \cdots    & \theta_{kd}g_d}}\\
\&F_{k+1} \rar{\phi_k} \&F_k
\end{tikzcd}\]
 for all $k \in \Z_d$. Notice that $\phi_k\theta_{(k+1)k} = \phi_k\phi_{k+1}\phi_{k+2}\cdots\phi_{k-1} = f\cdot1_{F_k} = f\theta_{kk}$ and $\phi_k\theta_{(k+1)i} = \theta_{ki}$ for all $i\neq k$. Therefore, we have that
\[\begin{split}
    \phi_k\rvect{\theta_{(k+1)1}g_1 & \cdots & \theta_{(k+1)d}g_d} &= \rvect{\theta_{k1}g_1 &\cdots &f\theta_{kk}g_k &\cdots &\theta_{kd}g_d}\\
     &=\rvect{\theta_{k1}g_1 &\cdots &\theta_{kk} g_k &\cdots  & \theta_{kd}g_d}D_k
\end{split}\] which implies the commutativity of the diagram as desired.

In order to show $\theta$ is an $S$-module isomorphism, let $(\alpha_1,\dots,\alpha_d) \in \Hom_{\MatFac{S}{d}{f}}(\mc P,X)$, $k \in \Z_d$, and denote the components of $\alpha_k=\rvect{\alpha_{k1} &\dots &\alpha_{kd}} \in \Hom_S(S^d,F_k)$ as above. By Lemma~\ref{thm:Hom_P_X}, $\alpha_{k(k+j)} = \theta_{k(k+j)}\alpha_{(k+j)(k+j)}$ for each $j\neq 0$. Hence, $\alpha_k = \rvect{\theta_{k1}\alpha_{11} & \cdots & \theta_{kd}\alpha_{dd}}$. It follows that the morphism $(\alpha_1,\dots,\alpha_d)$ depends only on the diagonal components $\alpha_{11},\alpha_{22},\dots,\alpha_{dd}$. In particular, the tuple $(\alpha_{11},\dots,\alpha_{dd})\in \bigoplus_{j=1}^dF_j$ is a pre-image for $(\alpha_1,\dots,\alpha_d)$ under the map $\theta$. Finally, $\theta$ is injective since $k$th component of $\rvect{\theta_{k1}g_1 & \cdots & \theta_{kk}g_k & \cdots &\theta_{kd}g_d}$ is $\theta_{kk}g_k = g_k$.
\end{proof}

\begin{cor}
For any $X \in \MatFac{S}{d}{f}$ of size $n$, the $\Gamma$-module $\mc F(X)=\Hom_{\MatFac{S}{d}{f}}(\mc P,X)$ is a free $S$-module of rank $dn$. In particular, $\mc F(X) \in \MCM(\Gamma)$. \qed
\end{cor}

Consider how the elements $e_{ii}\in \Gamma$ act on a morphism $\alpha: \mc P \to X$. From the Lemma \ref{thm:hom_P_X_free}, we may write $\alpha_k = \rvect{\theta_{k1}g_1 & \cdots & \theta_{kd}g_d}$ for some $(g_1,\dots,g_d) \in \bigoplus_{j=1}^d F_j$. For $i \in \Z_d$, we write $e_{ii} = (\epsilon_{ii}^1,\dots,\epsilon_{ii}^d)$ where $\epsilon_{ii}^k(a_1,\dots,a_d) = (0,\dots,0,a_i,0,\dots,0)$ for any $(a_1,\dots,a_d) \in S^d$. It follows that $\alpha_k \circ \epsilon_{ii}^k = \rvect{0 & \cdots & \theta_{ki}g_i &\cdots &0}$ where the only nonzero entry is in the $i$th position. Hence, $e_{ii}\cdot \alpha = \left(\rvect{0 & \cdots & \theta_{ki}g_i &\cdots &0}\right)_{k=1}^d \in e_{ii}\mc F(X)$ and in fact this is the form of every element of $e_{ii}\mc F(X)$: 
\[e_{ii}\mc F(X) = \left\{\left(\rvect{0 & \cdots & \theta_{ki}g_i &\cdots &0}\right)_{k=1}^d \Big\vert g_i \in F_i\right\}.\]

\begin{prop}\label{thm:H_dense}
Let $X=(\phi_1:F_2\to F_1,\dots,\phi_d:F_1\to F_d) \in \MatFac{S}{d}{f}$. Then, $(\theta\vert_{F_1},\dots,\theta\vert_{F_d}):X \to \mc H\mc F(X)$ is an isomorphism of matrix factorizations.
\end{prop}

\begin{proof}
First notice that for any $g_i \in F_i$, \[\theta\vert_{F_i}(g_i) = \theta(0,\dots,0,g_i,0,\dots,0) =\left(\rvect{0 & \cdots & \theta_{ki}g_i &\cdots &0}\right)_{k=1}^d \in e_{ii}\mc F(X).\] The restriction of $\theta$ is still injective and the preceding paragraph explains why $\theta$ restricted to $F_i$ is surjective. Therefore, it is enough to show the commutativity of the diagram
\[\begin{tikzcd}
F_1\arrow[r,"\phi_d"]\arrow[d,"\theta\vert_{F_1}"] &F_d\arrow[r,"\phi_{d-1}"]\arrow[d,"\theta\vert_{F_d}"] &\cdots \arrow[r,"\phi_2"] &F_2\arrow[r,"\phi_1"]\arrow[d,"\theta\vert_{F_2}"] &F_1\arrow[d,"\theta\vert_{F_1}"]\\
e_{11}\mc F(X)\arrow[r,swap,"z"] &e_{dd}\mc F(X)\arrow[r,swap,"z"] &\cdots \arrow[r,swap,"z"] &e_{22}\mc F(X) \arrow[r,swap,"z"] &e_{11}\mc F(X).
\end{tikzcd}\] Let $i \in \Z_d$ and $g_{i+1}\in F_{i+1}$. Then,
\(
    \theta\vert_{F_i}\phi_i(g_{i+1}) = \left(\rvect{0 & \cdots &\theta_{ki}\phi_i(g_{i+1}) &\cdots &0}\right)_{k=1}^d
\) and
\[\theta_{ki}\phi_i(g_{i+1}) =\begin{cases}
 fg_{i+1} &k=i+1\\
 \phi_k\phi_{k+1}\cdots\phi_{i-1}\phi_i(g_{i+1}) & k\neq i+1.
\end{cases}\]
To compute the other composition, recall that $z\cdot e_{(i+1)(i+1)}\alpha = e_{(i+1)i}\alpha$ for any $\alpha \in \mc F(X)$. Write $e_{(i+1)i} = (\epsilon_{(i+1)i}^1,\epsilon_{(i+1)i}^2,\dots,\epsilon_{(i+1)i}^d)$ where $\epsilon_{(i+1)i}^k:S^d \to S^d$ for each $k \in \Z_d$. It suffices to compute the composition of $S$-homomorphisms $S^d \to F_k$
\[\left(\rvect{0 \cdots &\theta_{k(i+1)}(g_i) \cdots &0} \circ \epsilon_{(i+1)i}^k\right)(a_1,\dots,a_d),\] for each $k \in \Z_d$ and $(a_1,\dots,a_d) \in S^d$. We have that 
\[\epsilon_{(i+1)i}^k(a_1,\dots,a_d) = \begin{cases}
(0,\dots,fa_i,\dots,0) &k = i+1\\
(0,\dots, a_i,\dots,0) &k\neq i+1
\end{cases}\] where the non-zero entries are in the $(i+1)$st position. Thus, the composition above is equal to $a_i\theta_{k(i+1)}(g_i)$ when $k\neq i+1$ and $fa_ig_i$ when $k=i+1$. Comparing this with the components of $\theta\vert_{F_i}\phi_i(g_{i+1})$ we conclude that $z\circ\theta\vert_{F_{i+1}}(g_{i+1}) = \theta\vert_{F_i}\circ\phi_i(g_{i+1})$. Hence, $(\theta\vert_{F_1},\dots,\theta\vert_{F_d}):X \to \mc H\mc F(X)$ is an isomorphism of matrix factorizations.
\end{proof}

As a consequence of Proposition \ref{thm:H_dense}, the functor $\mc H$ is also dense. This completes the proof of Theorem \ref{thm:MF_equiv_MCMgamma}. It is also worth mentioning that the analogous statement for the composition $\mc F\mc H$ is true, that is, $\mc F\mc H(M) \cong M$ for any $M \in \MCM(\Gamma)$. This follows from observing that the isomorphism of free $S$-modules, $\theta_{\mc H(M)}: M \to \mc F\mc H(M)$, is also a $\Gamma$-homomorphism. As in the proof of Proposition \ref{thm:H_full_faithful}, one can show that $e_{i(i-1)}\theta_{\mc H(M)}(m) = \theta_{\mc H(M)}(e_{i(i-1)}m)$ for all $m \in M$ and $i \in \Z_d$.

\section{The $d$-fold Branched Cover}\label{section:d_branched_cover}
    
Let $(S,\mf n, \bf k)$ be a complete regular local ring, and let $f$ be a non-zero non-unit in $S$. Set $R = S/(f)$ and fix an integer $d\ge 2$. 

\begin{defi}
The ($d$\textit{-fold) branched cover} of $R$ is the hypersurface ring
\[R^\sharp = S\llbracket z \rrbracket /(f+z^d).\]
\end{defi}

Throughout this section, we will also assume that $\bf k$ is algebraically closed and that the characteristic of $\bf k$ does not divide $d$. In this case, the polynomial $x^d -1 \in \mathbf{k}[x]$ has $d$ distinct roots in $\bf k$ and the group formed by its roots is cyclic of order $d$. Any generator of this group is a \textit{primitive $d$th root of unity}. Since $S$ is complete, it also contains primitive $d$th roots of $1 \in S$ \cite[Corollary A.31]{leuschke_cohen-macaulay_2012}. 

For the rest of this section, fix an element $\omega \in S$ such that $\omega^d = 1$ and $\omega^t \neq 1$ for all $0<t<d$. The ring $R^\sharp$ carries an automorphism $\sigma: R^\sharp \to R^\sharp$ of order $d$ which fixes $S$ and sends $z$ to $\omega z$. Denote by $R^\sharp[\sigma]$ the \textit{skew group algebra} of the cyclic group of order $d$ generated by $\sigma$ acting on $R^\sharp$. Specifically, $R^\sharp[\sigma] = \bigoplus_{i\in\Z_d}R^\sharp \cdot \sigma^i$ as $R^\sharp$-modules with multiplication given by the rule
\[(s\cdot \sigma^i)\cdot(t\cdot \sigma^j) = s\sigma^i(t) \cdot \sigma^{i+j}\]
for $s,t \in R^\sharp$ and $i,j \in \Z_d$. The left modules over $R^\sharp[\sigma]$ are precisely the $R^\sharp$-modules $N$ which carry a compatible action of $\sigma$, that is, an action of $\sigma$ such that $\sigma(rx) = \sigma(r)\sigma(x)$ for all $r \in R^\sharp$ and $x \in N$. It follows that $R^\sharp$ itself is naturally a left $R^\sharp[\sigma]$-module with the action of $\sigma$ given by evaluating $\sigma(r)$ for any $r \in R^\sharp$. We say that a left $R^\sharp[\sigma]$-module is \textit{maximal Cohen-Macaulay} (MCM as usual) if it is MCM when it is viewed as an $R^\sharp$-module. Denote the category of MCM $R^\sharp[\sigma]$-modules by $\MCM_{\sigma}(R^\sharp)$.

In the case $d=2$, Kn\"orrer showed that the category of MCM modules over $R^\sharp[\sigma]$ is equivalent to the category of matrix factorizations of $f$ with $2$ factors \cite[Proposition 2.1]{knorrer_cohen-macaulay_1987}. The main goal of this section is to extend the equivalence given by Kn\"orrer for all $d\ge 2$.

\begin{lem}\label{thm:decomp_of_N}
Let $N$ be an $R^\sharp[\sigma]$-module. Then, $N$ decomposes as an $S$-module into \(N = \bigoplus_{i \in \Z_d}N^{\omega^i}\) where
\[N^{\omega^i} = \left\{x \in N : \sigma(x) = \omega^i x\right\}.\] Furthermore, if $N$ is a MCM $R^\sharp[\sigma]$-module, then $N$ and each summand $N^{\omega^i}$ are finitely generated free $S$-modules.
\end{lem}

\begin{proof}
In order to justify the direct sum decomposition of $N$, we will make repeated use of the fact $\sum_{i=0}^{d-1}\omega^{ki} = 0$ for any $k \in \Z_d$. Let $x \in N$ and observe that
\[\begin{split}
    dx &= dx + \left(\sum_{i=0}^{d-1}\omega^{-i}\right)\sigma(x) + \left(\sum_{i=0}^{d-1}\omega^{-2i}\right)\sigma^2(x) + \cdots + \left(\sum_{i=0}^{d-1}\omega^{-(d-1)i}\right)\sigma^{d-1}(x)\\
    &= \sum_{i=0}^{d-1}\sigma^i(x) + \sum_{i=0}^{d-1}\omega^{-i}\sigma^i(x) + \cdots + \sum_{i=0}^{d-1}\omega^{-(d-1)i}\sigma^{i}(x)\\
    &=\sum_{k=0}^{d-1}\sum_{i=0}^{d-1}\omega^{-ik}\sigma^i(x).\\
\end{split}\] Also notice that, for any $k \in \Z_d$,
\[\begin{split}
    \sigma\left(\sum_{i=0}^{d-1}\omega^{-ik}\sigma^i(x)\right) &= \sigma(x) + \omega^{-k}\sigma^2(x) + \omega^{-2k}\sigma^3(x) + \cdots + \omega^{-(d-1)k}x\\
    &= \omega^k(x + \omega^{-k}\sigma(x) + \omega^{-2k}\sigma^2(x) + \cdots + \omega^{-(d-1)k}\sigma^{d-1}(x)).\\
\end{split}\] That is, $\sum_{i=0}^{d-1}\omega^{-ik}\sigma^i(x) \in N^{\omega^k}$. Since $\sigma$ is $S$-linear and $d$ is invertible in $S$, we have that \[x = \sum_{k=0}^{d-1}\sum_{i=0}^{d-1}\frac{\omega^{-ik}\sigma^i(x)}{d} \in N^1 + N^\omega + \cdots + N^{\omega^{d-1}}\] implying that $N = \sum_{k=0}^{d-1} N^{\omega^k}$.

Next, suppose we have a sum of elements
\begin{equation}\label{equation:lin_ind} x_0 + x_1 + \cdots + x_{d-1} = 0\end{equation} with $x_i \in N^{\omega^i}$ for each $i \in \Z_d$, and let $j \in \Z_d$. Notice that if $k,\ell \in \Z_d$, then $\omega^{-jk}\sigma^k(x_\ell) = \omega^{-jk + k\ell}x_\ell = \omega^{(-j+\ell)k}x_\ell$. In particular, $\omega^{-jk}\sigma^k(x_j) = x_j$ for all $k \in \Z_d$. Therefore, applying $\omega^{-jk}\sigma^k$ to \eqref{equation:lin_ind} gives us an equation
\[\omega^{-jk}x_0 + \omega^{(-j+1)k} x_1 + \cdots + x_j + \cdots + \omega^{(-j-1)k}x_{d-1}=0.\] 


Summing over $\Z_d$, we find that
\[\sum_{i\neq j}\sum_{k\in\Z_d}\omega^{k(-j+i)}x_i + dx_j = 0.\] Once again, since $\sum_{k=0}^{d-1}\omega^{k(-j+i)}=0$ for all $i\neq j$, we can conclude that $x_j=0$. Thus, $N= \bigoplus_{i=0}^{d-1}N^{\omega^i}$ as desired.

The second statement holds since a finitely generated $R^\sharp$-module $N$ is MCM over $R^\sharp$ if and only if it is free as an $S$-module.
\end{proof}

\begin{defprop} \label{thm:functors_A_B} Let $R, R^\sharp$, $R^\sharp[\sigma]$, and $\omega$ be as above. Let $\mu \in S$ be any root of $x^d + 1 \in S[x]$.
\begin{enumerate}[label= (\roman*)]
    \item Let $N$ be an MCM $R^\sharp[\sigma]$-module and $N^{\omega^i}$ be as in Lemma \ref{thm:decomp_of_N} for each $i\in \Z_d$. Define a matrix factorization $\mathcal A(N) \in \MatFac{S}{d}{f}$ as follows. Multiplication by $\mu z$ defines an $S$-linear homomorphism 
    \[N^{\omega^i} \to N^{\omega^{i+1}}\] for all $i \in \Z_d$. The composition 
    \[\begin{tikzcd}
    N^{\omega^{d-1}} \rar{\mu z} &N^{1} \rar{\mu z} &N^{\omega} \rar{\mu z} &\cdots \rar{\mu z} &N^{\omega^{d-2} \rar{\mu z}} &N^{\omega^{d-1}}
    \end{tikzcd}\] is equal to $-z^d = f$ times the identity on $N^{\omega^{d-1}}$. It follows that the above homomorphisms and free $S$-modules form a matrix factorization of $f$ in $\MatFac{S}{d}{f}$ which we denote as $\mc A(N)$. For a homomorphism $g: N \to M$ of MCM $R^\sharp[\sigma]$-modules, define a morphism of matrix factorizations \[\mc A(g) = \left(g\vert_{N^{\omega^{d-1}}},g\vert_{N^{\omega^{d-2}}},\dots,g\vert_{N^{1}}\right)\] where $g\vert_{N^{\omega^i}}$ denotes the restriction of $g$ to the $S$-direct summand $N^{\omega^i}$ of $N$. Thus, we have a functor
    \[\mc A: \MCM_{\sigma}(R^\sharp) \to \MatFac{S}{d}{f}.\]
    
    \item Let $X = (\phi_1:F_2 \to F_1,\dots,\phi_d:F_1 \to F_d) \in \MatFac{S}{d}{f}$. Define \[\mc B(X) = F_d\oplus F_{d-1} \oplus \cdots \oplus F_1\] as an $S$-module which has the structure of a $R^\sharp[\sigma]$-module by defining the action of $z$ as
    \[z \cdot (x_d,x_{d-1},\dots,x_1) = \left(\mu^{-1}\phi_d(x_1),\mu^{-1}\phi_{d-1}(x_d),\dots,\mu^{-1}\phi_1(x_2)\right)\] and the action of $\sigma$ as
    \[\sigma\cdot(x_d,x_{d-1},\dots,x_1) = (x_d,\omega x_{d-1},\omega^2 x_{d-2},\dots,\omega^{d-1}x_1),\]
    for any $x_i \in F_i$, $i\in\Z_d$. For a morphism of matrix factorizations $\alpha =(\alpha_1,\alpha_2,\dots,\alpha_d): X \to X'$, where $X' = (\phi_1':F_2'\to F_1',\dots, \phi_d':F_1'\to F_d')$, define $\mc B(\alpha): \mc B(X) \to \mc B(X')$ by 
    \[\mc B(\alpha)(x_d,x_{d-1},\dots,x_1) = (\alpha_d(x_d),\alpha_{d-1}(x_{d-1}),\dots,\alpha_1(x_1))\] for all $(x_d,x_{d-1}\dots,x_1) \in \mc B(X)$. Thus, we have a functor 
    \[\mc B: \MatFac{S}{d}{f} \to \MCM_{\sigma}(R^\sharp).\]
\end{enumerate}
\end{defprop}

\begin{proof} Several pieces of the definitions need justification. First we note that, since $-1$ has a $d$th root in $\bf k$, we may apply \cite[Corollary A.31]{leuschke_cohen-macaulay_2012} to obtain an element $\mu \in S$ such that $\mu^d = -1$.
\begin{enumerate}[label =(\roman*)] 
    \item Multiplication by $\mu z$ defines an $S$-linear map $N^{\omega^i} \to N^{\omega^{i+1}}$ for any $i \in \Z_d$ since $\mu \in S$ and
    \[\sigma( z x) = \sigma( z)\sigma(x) = \omega^{i+1} z x\] for all $x \in N^{\omega^i}$. Notice that $(\mu z)^d = \mu^d z^d = f \in R^\sharp$. Therefore, the composition
    \[\begin{tikzcd}
    N^{\omega^{i+1}} \rar{(\mu z)^{d-1}} &N^{\omega^{i}} \rar{\mu z} &N^{\omega^{i+1}}
    \end{tikzcd}\] equals $f\cdot 1_{N^{\omega^i}}$ for all $i \in \Z_d$. We know that each $N^{\omega^i}$ is a free $S$-module by Lemma \ref{thm:decomp_of_N} so, by applying Lemma \ref{thm:n_equals_m}, we have that $\rank_S(N^{\omega^i}) = \rank_S(N^{\omega^{i+1}})$ for all $i \in \Z_d$. This implies that $\mc A(N) \in \MatFac{S}{d}{f}$. 
    
    If $g: N \to M$ is a homomorphism of $R^\sharp[\sigma]$-modules, then $g(rx) = rg(x)$ and $\sigma g(x) = g(\sigma(x))$ for all $x \in N$ and $r\in R^\sharp$. It follows that $g\vert_{N^{\omega^i}}(N^{\omega^i})\subseteq M^{\omega^i}$ and the diagram
    \[\begin{tikzcd}
    N^{\omega^i} \dar[swap]{g\vert_{N^{\omega^i}}} \rar{\mu z} &N^{\omega^{i+1}} \dar{g\vert_{N^{\omega^{i+1}}}}\\
    M^{\omega^i} \rar{\mu z} &M^{\omega^{i+1}}
    \end{tikzcd}\] commutes for all $i \in \Z_d$. In other words, $\mc A(g)$ is a morphism of matrix factorizations $\mc A(N) \to \mc A(M)$.
    \item First we justify that the defined action of $z$ and $\sigma$ make $\mc B(X)$ a MCM $R^\sharp[\sigma]$-module. Recall the homomorphisms $\theta_{ki}^X: F_i \to F_k$, from Section \ref{section:Frobenius}, which are given by
    \[\theta_{ki}^X =\begin{cases}
                    1_{F_k} & i=k\\
                    \phi_k\phi_{k+1}\cdots\phi_{i-2}\phi_{i-1} & i\neq k.\\
\end{cases}\] We will drop the superscript $X$ for the rest of this proof. We claim that for any $s\ge 1$ and $(x_d,x_{d-1},\dots,x_1) \in \mc B(X)$,
    \[z^s \cdot(x_d,x_{d-1},\dots,x_1) = f^q\mu^{-r}(\theta_{d(d+r)}(x_{d+r}),\dots, \theta_{1(1+r)}(x_{1+r}))\] where $s = dq + r$, $q \ge 0$, and $0 \leq r < d$.
    When $s=1$ the formula is precisely the defined action of $z$ on $\mc B(X)$. Assume the claim is true for $s=dq + r \ge 1$ with $q \ge 0$ and $0\leq r < d$ and consider multiplication by $z^{s+1}$. By induction we have that
    \[\begin{split}z^{s+1}\cdot (x_d,\dots ,x_1) &= z\cdot f^q\mu^{-r}((\theta_{d(d+r)}(x_{d+r}),\dots, \theta_{1(1+r)}(x_{1+r}))\\
    &= f^q\mu^{-(r+1)}(\phi_d\theta_{1(1+r)}(x_{1+r}),\dots,\phi_1\theta_{2(2+r)}(x_{2+r})).
    \end{split}\] If $r = d-1$, then $\phi_{k-1}\theta_{k(k+r)} = f \cdot 1_{F_{k-1}}$ for each $k \in \Z_d$ and therefore
    \[z^{s+1} = f^{q+1}\mu^{-(r+1)}(x_{d},x_{d-1},\dots,x_1).\]
    If $0 \leq r < d-1$, then $0\leq r+1 < d$ and therefore $\phi_{k-1}\theta_{k(k+r)} = \theta_{(k-1)(k+r)}$ for each $k \in \Z_d$. In this case,
    \[z^{s+1} = f^q\mu^{-(r+1)}(\theta_{d(1+r)}(x_{1+r}),\dots,\theta_{1(2+r)}(x_{2+r}))\] which completes the induction. It follows that multiplication by $z^d$ is given by
    \[ z^d \cdot (x_d,\dots,x_1)  = f\mu^{-d}(x_d,\dots,x_1).\] By definition, $\mu^{-d} = -1$. Thus, $(f+z^d)\mc B(X) = 0$, that is, $\mc B(X)$ is an $R^\sharp$-module. In fact, since $\mc B(X)$ is free as an $S$-module, it is MCM as an $R^\sharp$-module.
    
    In order to show that $\mc B(X)$ has the structure of an $R^\sharp[\sigma]$-module, we must show that $\sigma(rx)=\sigma(r)\sigma(x)$ for all $r \in R^\sharp$ and $x \in \mc B(X)$. It suffices to show that $\sigma(zx) = \sigma(z)\sigma(x)$ for all $x \in \mc B(X)$. This follows since
    \[\begin{split}
        \sigma(z)\sigma(x) &= \omega z \cdot (x_d,\omega x_{d-1},\dots,\omega^{d-1}x_1)\\
        &= z \cdot (\omega x_d, \omega^2 x_{d-1},\dots, x_1)\\
        &= (\mu^{-1} \phi_d(x_1), \mu^{-1}\omega\phi_{d-1}(x_d),\dots, \mu^{-1}\omega^{d-1}\phi_1(x_2))\\
        &= \sigma\left(\mu^{-1}\phi_d(x_1), \mu^{-1}\phi_{d-1}(x_d),\dots,\mu^{-1}\phi_1(x_2)\right)\\
        &= \sigma(zx)
    \end{split}\] for any $x=(x_d,x_{d-1},\dots,x_1) \in \mc B(X)$. Hence, $ \mc B(X) \in \MCM_{\sigma}(R^\sharp)$.
    
    Finally, we must show that $\mc B(\alpha)$ forms a homomorphism of $R^\sharp[\sigma]$-modules. This is straightforward to verify by recalling that $\alpha_k\phi_k = \phi_k'\alpha_{k+1}$ for all $k \in \Z_d$.

\end{enumerate}

\end{proof}

\begin{thm}\label{thm:R_hash_sigma_equiv_MFd}
The functors $\mc A : \MCM_{\sigma}(R^\sharp) \to \MatFac{S}{d}{f}$ and $\mc B: \MatFac{S}{d}{f} \to \MCM_{\sigma}(R^\sharp)$ are naturally inverse and establish an equivalence of the categories $\MCM_\sigma(R^\sharp)\approx \MatFac{S}{d}{f}$.
\end{thm}

\begin{proof}
Let $X = (\phi_1:F_2 \to F_1,\dots,\phi_d:F_1 \to F_d) \in \MatFac{S}{d}{f}$. Then, $\mc B(X) = F_d \oplus F_{d-1}\oplus \cdots \oplus F_1$ with the action of $\sigma$ on $\mc B(X)$ given by $\sigma(x_d,\dots,x_1) = (x_d,\omega x_{d-1},\dots,\omega^{d-1}x_1)$ for each $x_i \in F_i$. For each $i\in \Z_d$, the $S$-module $F_i$ is embedded into $\mc B(X)$ via the natural inclusion map which we will denote as 
\[q_i : F_i \to \mc B(X)\]
Notice that the action of $\sigma$ on $\mc B(X)$ implies that
\[\mc B(X)^{\omega^{d-i}} = \{(0,\dots,0,x_i,0,\dots,0) : x_i \in F_i\} = q_i(F_i).\] Therefore, the matrix factorization $\mc A\mc B(X)$ is given by
\[\begin{tikzcd}
\mc B(X)^{\omega^{d-1}} \rar{\mu z} &\mc B(X)^{1} \rar{\mu z} & \cdots \rar{\mu z} &\mc B(X)^{\omega^{d-2}} \rar{\mu z} &\mc B(X)^{\omega^{d-1}}
\end{tikzcd}\] which is isomorphic to $X$ via the morphism
\[\begin{tikzcd}
F_1 \dar{q_1} \rar{\phi_d} &F_d \dar{q_d} \rar{\phi_{d-1}} &\cdots \rar{\phi_2} &F_2 \dar{q_2} \rar{\phi_1} &F_1 \dar{q_1}\\
\mc B(X)^{\omega^{d-1}} \rar{\mu z} &\mc B(X)^{1} \rar{\mu z} & \cdots \rar{\mu z} &\mc B(X)^{\omega^{d-2}} \rar{\mu z} &\mc B(X)^{\omega^{d-1}}. 
\end{tikzcd}\] Explicitly, the diagram commutes since if $k \in \Z_d$ and $x \in F_{k+1}$, then
\[\mu z q_{k+1}(x) = \mu q_{k}(\mu^{-1}\phi_k(x)) = q_k\phi_k(x).\]

To show $\mc A\mc B$ is naturally isomorphic to the identity, suppose we have a morphism $\alpha= (\alpha_1,\dots,\alpha_d): X \to X'$ where $X' = (\phi_1':F_2'\to F_1', \dots,\phi_d':F_1'\to F_d') \in \MatFac{S}{d}{f}$. The matrix factorizations $X'$ is isomorphic to $\mc A\mc B(X')$ via the morphism $(q_1',q_2',\dots,q_d')$ where $q_i': F_i' \to \mc B(X')$ is the natural inclusion. Recall that the homomorphism $\mc B(\alpha)$ is given by
\[\mc B(\alpha)(x_d,x_{d-1},\dots,x_1) = (\alpha_d(x_d),\alpha_{d-1}(x_{d-1}),\dots,\alpha_1(x_1)).\] Applying the functor $\mc A$ forms a morphism of matrix factorizations by restricting $\mc B(\alpha)$ to the submodules $\mc B(X)^{\omega^{d-i}}$. The images of these restrictions land in the submodules $\mc B(X')^{\omega^{d-i}}$. In other words, the $k$th component of the morphism $\mc A\mc B(\alpha)$ is given by the composition
\[\begin{tikzcd}
\mc B(X)^{\omega^{d-k}} \rar{p_k} &F_k \rar{\alpha_k} &F_k' \rar{q_k'} &\mc B(X')^{\omega^{d-k}}
\end{tikzcd}\] where $p_k$ is the natural projection onto $F_k$. Therefore,
\[\begin{split}
    \mc A \mc B(\alpha) \circ (q_1,q_2,\dots, q_d) & = (q_1'\alpha_1p_1,q_2'\alpha_2p_2,\dots,q_d'\alpha_dp_d)\circ (q_1,q_2,\dots,q_d)\\
    &= (q_1'\alpha_1,q_2'\alpha_2,\dots,q_d'\alpha_d)
\end{split}\] and this implies the commutativity of the diagram
\[\begin{tikzcd}
X \rar{\alpha} \dar[swap]{(q_1,\dots,q_d)} &X' \dar{(q_1',\dots,q_d')}\\
\mc A\mc B(X) \rar{\mc A \mc B(\alpha)} &\mc A \mc B(X').
\end{tikzcd}\]

Next, let $N$ be an MCM $R^\sharp[\sigma]$-module. As an $S$-module, \[\mc B\mc A(N) = N^1 \oplus N^{\omega} \oplus \cdots \oplus N^{\omega^{d-1}}\] and in fact the natural $S$-isomorphism $\Psi_N: \mc B \mc A(N) \to N$ given by $(n_0,n_1,\dots,n_{d-1}) \mapsto \sum_{i \in \Z_d} n_i$ is also an $R^\sharp[\sigma]$-homomorphism. To see this, let $(n_0,n_1,\dots,n_{d-1}) \in \mc B\mc A(N)$. Then, $\Psi_N$ is a $R^\sharp$-homomorphism since
\[\begin{split}
\Psi_N(z\cdot(n_0,n_1,\dots,n_{d-1})) &= \Psi_N(\mu^{-1}\mu z n_{d-1}, \mu^{-1}\mu z n_0,\dots,\mu^{-1}\mu z n_{d-1})\\
&= \Psi_N(zn_{d-1},zn_0,\dots,zn_{d-1})\\
&= z(n_0+n_1+\cdots n_{d-1})\\
&=z\Psi_N(n_0,n_1,\dots,n_{d-1})
\end{split}\] and a $R^\sharp[\sigma]$-homomorphism since
\[\begin{split}
    \Psi_N(\sigma(n_0,n_1,\dots,n_{d-1})) &= \Psi_N(n_0,\omega n_1,\dots, \omega^{d-1}n_{d-1})\\
    &= n_0 + \omega n_1 + \cdots + \omega^{d-1}n_{d-1}\\
    &= \sigma(n_0) + \sigma(n_1) + \cdots + \sigma(n_{d-1})\\
    &=\sigma(n_0 + n_1 + \cdots n_{d-1})\\
    &=\sigma(\Psi_N(n_0,n_1,\dots,n_{d-1})).
\end{split}\]

\end{proof}

In Section \ref{section:KRS} we showed that the category of MCM modules over the endomorphism ring of the projective object $\mc P=\mc P_1\oplus \mc P_2\oplus \cdots \oplus \mc P_d$ is equivalent to the category of matrix factorizations of $f$ with $d\ge 2$ factors. Together with Proposition \ref{thm:R_hash_sigma_equiv_MFd} we have an induced equivalence of the module categories $\MCM_\sigma(R^\sharp)$ and $\MCM(\End_{\MatFac{S}{d}{f}}(\mc P)^{\textup{op}})$. In fact, the two rings are isomorphic which we will see below. Of course, both statements require that the characteristic of $\bf k$ does not divide $d$. Recall, also from Section \ref{section:KRS}, that $\Gamma=\End_{\MatFac{S}{d}{f}}(\mc P)^{\textup{op}}$ is a free $S$-module with basis given by the elements $\{e_{ij}\}_{i,j \in \Z_d}$. The main rules for multiplication in $\Gamma$ are given in Lemmas \ref{thm:properties_of_eij}, \ref{thm:z^s}, and \ref{thm:e_ij=prod}.

\begin{prop}\label{thm:R_sharp_sigma_iso_gamma}
The rings $R^\sharp[\sigma]$ and $\Gamma = \End_{\MatFac{S}{d}{f}}(\mc P)^{\textup{op}}$ are isomorphic.
\end{prop}

\begin{proof}
The set $\{z^i\sigma^j\}_{i,j\in\Z_d}$ forms a basis for $R^\sharp[\sigma]$ over $S$. As in \ref{thm:functors_A_B}, let $\mu \in S$ be any root of $x^d+1 \in S[x]$. Define a map $\psi: R^\sharp[\sigma] \to \Gamma$ by $\psi(z)= \mu \sum_{i\in \Z_d}e_{i(i-1)}$ and $\psi(\sigma) =\sum_{i\in\Z_d}\omega^{-i}e_{ii}$. Extend $\psi$ multiplicatively, that is, define $\psi(z^i\sigma^j) = \psi(z)^i\psi(\sigma)^j$ for all $i,j \in \Z_d$. Since $\{z^i\sigma^j\}_{i,j\in\Z_d}$ is an $S$-basis, $\psi$ extends uniquely to a well defined $S$-linear homomorphism. Using the fact that $\sigma z = \sigma(z) \sigma = \omega z \sigma \in R^\sharp[\sigma]$, it is not hard to see that $\psi$ is also a homomorphism of rings.

As $S$-modules, both $R^\sharp[\sigma]$ and $\Gamma$ are free of rank $d^2$. Therefore, to conclude that $\psi$ is an isomorphism, it suffices to check surjectivity. First, we show that the element $e_{kk}$ is in the image of $\psi$ for each $k \in \Z_d$. Indeed, if $j \in \Z_d$, then
\[\psi(\sigma^j) = \sum_{i\in \Z_d}\omega^{-ji}e_{ii}.\] Thus, for any $k \in \Z_d$,
\[\begin{split}
    \psi\left(\frac{1}{d}\sum_{j\in \Z_d}\omega^{jk}\sigma^j\right) &= \frac{1}{d}\sum_{j\in \Z_d}\omega^{jk}\psi(\sigma)^j\\
    &=\frac{1}{d}\sum_{j\in \Z_d}\sum_{i\in\Z_d}\omega^{j(k-i)}e_{ii}\\
    &=\frac{1}{d}\sum_{i\neq k}\sum_{j\in \Z_d}\omega^{j(k-i)}e_{ii} + \frac{1}{d}\sum_{j\in \Z_d}e_{kk}\\
    &=e_{kk}.
\end{split}\]
Hence, the elements $e_{11},e_{22},\dots,e_{dd},$ and $\sum_{i \in \Z_d}e_{i(i-1)}$ are in the image of $\psi$. It follows that $e_{k(k-1)} \in \Ima\psi$ for all $k \in \Z_d$ since $e_{kk}\sum_{i\in\Z_d}e_{i(i-1)} = e_{k(k-1)}$ by Lemma \ref{thm:properties_of_eij} \ref{thm:properties_of_eij:part_4}. Finally, Lemma \ref{thm:e_ij=prod} allows us to conclude that $e_{ij} \in \Ima \psi$ for all $i,j \in \Z_d$, implying that $\psi$ is surjective as desired.
\end{proof}

\section{The syzygy of a matrix factorization}\label{section:OmegaMF}

Let $(S,\mf n,\bf k)$ be a complete regular local ring, $f \in S$ a non-zero non-unit, and set $R=S/(f)$. For a matrix factorization $X = (\phi_1:F_2\to F_1,\dots,\phi_d :F_1 \to F_d)$, we study its syzygy, $\Omega_{\MatFac{S}{d}{f}}(X) = (\Omega_1,\Omega_2,\dots,\Omega_d)$, defined in Section \ref{section:Frobenius}. Note that, by Section \ref{section:KRS}, the additional assumption of $S$ being complete allows us to utilize the Krull-Remak-Schmidt Theorem in $\MatFac{S}{d}{f}$. As in previous sections, let $\mc P = \bigoplus_{i \in \Z_d}\mc P_i$ and $\Gamma = \End_{\MatFac{S}{d}{f}}(\mc P)^{\textup{op}}$.

The exact structure on $\MatFac{S}{d}{f}$ ensures that $\Omega_{MF_S^d(f)}(X)$ is \textit{stably equivalent} to any matrix factorization $K$ such that there exists a short exact sequence \[\begin{tikzcd}K \rar[tail] &P \rar[two heads] &X\end{tikzcd}\] where $P$ is projective. This follows from the appropriate version of Schanuel's Lemma in $\MatFac{S}{d}{f}$.

\begin{lem}(Schanuel's Lemma) \label{thm:Schanuel}
Let $X \in \MatFac{S}{d}{f}$ and suppose
\[\begin{tikzcd}
&K \rar[tail]{q} &P \rar[two heads]{p} &X
\end{tikzcd} \qquad\text{and} \begin{tikzcd}
&K' \rar[tail]{q'} &P' \rar[two heads]{p'} &X
\end{tikzcd}\] are short exact sequences of matrix factorizations with $P$ and $P'$ projective. Then, $P \oplus K' \cong K \oplus P'$. \qed
\end{lem}

We omit the proof as it follows from \cite[Proposition 2.12]{buhler_exact_2010}. The next Lemma follows directly from Lemma \ref{thm:Schanuel}, KRS, and the sequences \eqref{equation:enough_proj} and \eqref{equation:enough_inj}.




\begin{lem}\label{thm:omega_additive}
Let $X,X' \in \MatFac{S}{d}{f}$. Then,
\begin{enumerate}
    \item $\Omega_{\MatFac{S}{d}{f}}(X\oplus X') \cong \Omega_{\MatFac{S}{d}{f}}(X) \oplus \Omega_{\MatFac{S}{d}{f}}(X')$
    \item $\Omega_{\MatFac{S}{d}{f}}^-(X \oplus X') \cong \Omega_{\MatFac{S}{d}{f}}^-(X) \oplus \Omega_{\MatFac{S}{d}{f}}^-(X')$
    \item $\Omega_{\MatFac{S}{d}{f}}(X)$ (respectively $\Omega_{\MatFac{S}{d}{f}}^-(X)$) is projective if and only if $X$ is projective.
\end{enumerate} \qed
\end{lem}

As a consequence, both $\Omega_{\MatFac{S}{d}{f}}(-)$ and $\Omega_{\MatFac{S}{d}{f}}^-(-)$ define additive functors from the stable category $\underline{\textup{MF}}_S^d(f)$ to itself.

\begin{prop}\label{thm:omega_iso_omega-1}
Let $X \in \MatFac{S}{d}{f}$ be of size $n$ and $\mc P = \bigoplus_{i=1}^d \mc P_i$. Then, we have isomorphisms
\begin{enumerate}[itemsep = 2pt, label = (\roman*)]
    \item \label{thm:omega_iso_omega-1:part1} $\Omega_{\MatFac{S}{d}{f}}(\Omega_{\MatFac{S}{d}{f}}^{-}(X)) \cong X \oplus \mc P^{(d-2)n}$,
    \item \label{thm:omega_iso_omega-1:part2} $\Omega_{\MatFac{S}{d}{f}}^{-}(\Omega_{\MatFac{S}{d}{f}}(X)) \cong X \oplus \mc P^{(d-2)n}$, and
    \item \label{thm:omega_iso_omega-1:part3} $\Omega_{\MatFac{S}{d}{f}}(X) \cong \Omega_{\MatFac{S}{d}{f}}^{-}(X)$.
\end{enumerate}
\end{prop}

\begin{proof} Since $\Omega_{\MatFac{S}{d}{f}}^{-}(X)$ is of size $(d-1)n$, there is a short exact sequence
\[\begin{tikzcd}
\Omega_{\MatFac{S}{d}{f}}(\Omega_{\MatFac{S}{d}{f}}^{-}(X)) \rar[tail] &\mc P^{(d-1)n} \rar[two heads] &\Omega_{\MatFac{S}{d}{f}}^{-}(X).
\end{tikzcd}\] By applying Schanuel's lemma to this sequence and the sequence \eqref{equation:enough_inj}, we find that
\[\Omega_{\MatFac{S}{d}{f}}(\Omega_{\MatFac{S}{d}{f}}^{-}(X)) \oplus \mc P^n \cong X \oplus \mc P^{(d-1)n}.\] We may cancel one copy of $\mc P^n$ from both sides by KRS to obtain the first statement. Dually, the second statement follows from the injective version of Schanuel's Lemma.

In order to prove \ref{thm:omega_iso_omega-1:part3}, we construct an explicit isomorphism. For each $k \in \Z_d$ define an $S$-homomorphism $\alpha_k: \widehat F_k \to \widehat F_k$ by 
    \[\alpha_k = \begin{pmatrix}
1_{F_{k+1}} &\theta_{(k+1)(k+2)}^X &\theta_{(k+1)(k+3)}^X &\cdots &\theta_{(k+1)(k-1)}^X\\
0           &1_{F_{k+2}}           &\theta_{(k+2)(k+3)}^X &\cdots &\theta_{(k+2)(k-1)}^X\\
0           &0                     &1_{F_{k+3}}           &\ddots &\vdots\\
\vdots      &\vdots                &\vdots                &\ddots &\theta_{(k-2)(k-1)}^X\\
0           &0                     &\cdots                &       &1_{F_{k-1}}
\end{pmatrix}.\] It is not hard to see that each $\alpha_k$ is an isomorphism and that the diagram 
\[\begin{tikzcd}
\widehat F_{k+1} \dar{\alpha_{k+1}}\rar{\Omega_k} &\widehat F_k \dar{\alpha_k}\\
\widehat F_{k+1} \rar{\Omega_k^{-}}               &\widehat F_k
\end{tikzcd}\] commutes for all $k \in \Z_d$. Hence, we have an isomorphism of matrix factorizations $(\alpha_1,\dots,\alpha_d): \Omega_{\MatFac{S}{d}{f}}(X) \to \Omega_{\MatFac{S}{d}{f}}^{-}(X)$.
\end{proof}

\begin{remark} In the case $d=2$, no projective summands occur in the first two isomorphisms. This agrees with what we saw in Example \ref{example:d=2_3} which said that the syzygy and cosyzygy operations are isomorphic to the shift functor:
\[\Omega_{\MatFac{S}{2}{f}}(\phi,\psi) \cong (\psi,\phi) \cong \Omega_{\MatFac{S}{2}{f}}^-(\phi,\psi)\] for any $(\phi,\psi)\in \MatFac{S}{2}{f}$. From this we can see that all three statements of Proposition \ref{thm:omega_iso_omega-1} are immediate when $d=2$. In particular, the isomorphism in Proposition \ref{thm:omega_iso_omega-1} \ref{thm:omega_iso_omega-1:part3} is actually the identity. In contrast, the isomorphism constructed in Proposition \ref{thm:omega_iso_omega-1} \ref{thm:omega_iso_omega-1:part3} when $d=3$ is
\[\begin{tikzcd}[ampersand replacement=\&, column sep = 4em, row sep= 4em]
F_2\oplus F_3 \dar{\begin{pmatrix}1_{F_2} & \phi_2\\ 0 &1_{F_3}\\ \end{pmatrix}} \rar{\begin{pmatrix}
-\phi_1 & -\phi_1\phi_2\\
1_{F_2} &0\\
\end{pmatrix}} \&F_1\oplus F_2 \dar{\begin{pmatrix}1_{F_1} & \phi_1\\ 0 &1_{F_2}\\ \end{pmatrix}} \rar{\begin{pmatrix}
-\phi_3 & -\phi_3\phi_1\\
1_{F_1} &0\\
\end{pmatrix}} \&F_3\oplus F_1 \dar{\begin{pmatrix}1_{F_3} & \phi_3\\ 0 &1_{F_1}\\ \end{pmatrix}} \rar{\begin{pmatrix}
-\phi_2 & -\phi_2\phi_3\\
1_{F_3} &0\\
\end{pmatrix}} \&F_2\oplus F_3 \dar{\begin{pmatrix}1_{F_2} & \phi_2\\ 0 &1_{F_3}\\ \end{pmatrix}}\\
F_2\oplus F_3 \rar[swap]{\begin{pmatrix}
0 & -\phi_1\phi_2\\
1_{F_2} &-\phi_2\\
\end{pmatrix}} \&F_1\oplus F_2 \rar[swap]{\begin{pmatrix}
0 & -\phi_3\phi_1\\
1_{F_1} &-\phi_1\\
\end{pmatrix}} \&F_3\oplus F_1 \rar[swap]{\begin{pmatrix}
0 & -\phi_2\phi_3\\
1_{F_3} &-\phi_3\\
\end{pmatrix}} \&F_2\oplus F_3
\end{tikzcd}.\]
\end{remark}

The next Proposition uses Proposition \ref{thm:omega_iso_omega-1} to show that each object in $\MatFac{S}{d}{f}$ has a projective resolution which is periodic of period at most 2. A projective resolution in this context is with respect to the exact structure on $\MatFac{S}{d}{f}$ (see \cite[12.1]{buhler_exact_2010}).

\begin{prop}\label{thm:periodicity_in_dMF}
Let $X \in \MatFac{S}{d}{f}$. Then, $X$ has a projective resolution which is periodic with period at most 2:
\[\begin{tikzcd}
\cdots \rar{q} &I(X) \rar{p} &P(X) \rar{q} &I(X) \rar{p} &P(X) \rar[two heads] &X.
\end{tikzcd}\]
\end{prop}

\begin{proof} Set $\Omega(X) = \Omega_{\MatFac{S}{d}{f}}(X)$ and $\Omega^-(X) = \Omega^-_{\MatFac{S}{d}{f}}(X)$. Let $\alpha: \Omega(X) \to \Omega^-(X)$ be the isomorphism constructed in Proposition \ref{thm:omega_iso_omega-1}. Then, we have two diagrams

\begin{equation}\label{diagram:exactness_1}\begin{tikzcd}
I(X)\drar[two heads,swap]{\alpha^{-1}\eta^X} \ar[rr] & &P(X) \drar[two heads, swap]{\rho^X} \ar[rr] && I(X)\\
&\Omega(X) \urar[swap, tail]{\epsilon^X} & &X \urar[tail,swap]{\lambda^X}
\end{tikzcd}\end{equation}
and
\begin{equation}\label{diagram:exactness_2}\begin{tikzcd}
P(X)\drar[two heads,swap]{\rho^X} \ar[rr] & &I(X) \drar[two heads, swap]{\alpha^{-1}\eta^X} \ar[rr] && P(X)\\
&X \urar[swap, tail]{\lambda^X} & &\Omega(X) \urar[tail,swap]{\epsilon^X}
\end{tikzcd}.\end{equation}
The middle sequence in \eqref{diagram:exactness_2} is short exact since we have a commutative diagram
\[\begin{tikzcd}
X \dar[equals] \rar[tail]{\lambda^X} &I(X) \dar[equals] \rar[two heads]{\alpha^{-1}\eta^X} &\Omega(X) \dar{\alpha}\\
X \rar[tail]{\lambda^X} &I(X) \rar[two heads]{\eta^X} &\Omega^-(X)
\end{tikzcd}\] with vertical isomorphisms.
The desired resolution follows by splicing together \eqref{diagram:exactness_1} and \eqref{diagram:exactness_2}, that is, by setting $p=\epsilon^X\alpha^{-1}\eta^X$ and $q= \lambda^X\rho^X$.
\end{proof}

\begin{cor}\label{thm:periodicity_in_MCM_Gamma}
Let $\mc P = \bigoplus_{i \in \Z_d} \mc P_i$ and $\Gamma = \End_{\MatFac{S}{d}{f}}(\mc P)^{\textup{op}}$. Then, every finitely generated left $\Gamma$-module has a projective resolution which is eventually periodic of period at most 2.
\end{cor}

\begin{proof}
Let $N$ be a finitely generated $\Gamma$-module and set $r = \dim S$. Let $M = \syz_\Gamma^r(N)$ be an arbitrary $r$th syzygy of $N$ over $\Gamma$, and let
\[\begin{tikzcd}[column sep = small]
0 \rar &M \rar &P_{r-1} \rar &P_{r-2} \rar &\cdots \rar &P_1 \rar &P_0 \rar &N \rar &0
\end{tikzcd}\] be the first $r-1$ steps of a projective resolution of $N$ for some finitely generated projective $\Gamma$-modules $P_i$, $i=0,1,\dots,r-1$. Recall that finitely generated projective $\Gamma$-modules are in $\MCM(\Gamma)$, that is, they are finitely generated free $S$-modules. Thus, the Depth Lemma implies that $\dep_S(M) = r$. Since MCM $S$-modules are free, we have that $M \in \MCM(\Gamma)$ as well. Now, by Section \ref{section:KRS}, there exists $X \in \MatFac{S}{d}{f}$ of size $n$ such that $\mc F(X) = \Hom_{\MatFac{S}{d}{f}}(\mc P,X) \cong M$. Since $\mc P$ is projective in $\MatFac{S}{d}{f}$, the functor $\mc F$ is exact. In particular, applying $\mc F$ to the periodic resolution constructed in Proposition \ref{thm:periodicity_in_dMF} yields an exact sequence of MCM $\Gamma$-modules
\[\begin{tikzcd}
\cdots \rar{\mc F(p)} &\mc F(P(X)) \rar{\mc F(q)} &\mc F(I(X)) \rar{\mc F(p)} &\mc F(P(X)) \rar &M \rar &0.
\end{tikzcd}\] Actually, this is a free resolution of $M$ over $\Gamma$ since $\mc F(P(X)) \cong \mc F(\bigoplus_{i \in \Z_d}\mc P_i^n) \cong \Gamma^n$ and similarly $\mc F(I(X)) \cong \Gamma^n$. Thus, splicing together this periodic free resolution of $M$ and the projective resolution of $N$, we get an eventually periodic resolution of $N$.
\end{proof}

Proposition \ref{thm:periodicity_in_MCM_Gamma} and Lemma \ref{thm:inj_dim_Gamma} below give a homological description of $\Gamma$ which resembles that of a (commutative) hypersurface ring. Recall that a Noetherian ring $\Lambda$ is said to be \textit{Iwanaga-Gorenstein} if $\textup{injdim}_\Lambda\Lambda$ and $\textup{injdim}_{\Lambda^{\textup{op}}}\Lambda$ are both finite.

\begin{lem}\label{thm:inj_dim_Gamma}
Let $r$ be the Krull dimension of $S$. The ring $\Gamma = \End_{\MatFac{S}{d}{f}}(\mc P)^{\textup{op}}$ is Iwanaga-Gorenstein of dimension $r$.
\end{lem}

\begin{proof} First we show that any short exact sequence
\begin{equation}\label{equation:Gamma_relative_injective}\begin{tikzcd}
0 \rar &\Gamma \rar{q} &M \rar{p} &M' \rar &0
\end{tikzcd}\end{equation} with $M,M' \in \MCM(\Gamma)$ splits. To see this, first note that $\mc H(\Gamma) = \mc H\mc F(\mc P) \cong \mc P$ is injective in $\MatFac{S}{d}{f}$. Therefore, the short exact sequence of matrix factorizations
\[\begin{tikzcd}
\mc H(\Gamma) \rar[tail]{\mc H(q)} &\mc H(M) \rar[two heads]{\mc H(p)} &\mc H(M')
\end{tikzcd}\] is split. Since $\mc H$ is full and faithful, there exists $t:M \to \Gamma$ such that $\mc H(t)\mc H(q) = 1_{\mc H(\Gamma)}$ and $tq = 1_\Gamma$ which implies that \eqref{equation:Gamma_relative_injective} is split.

To finish the proof, we apply results from \cite{auslander_isolated_1986} which apply to both $\Gamma$ and $\Gamma^{\textup{op}}$. By \cite[Lemma 1.1]{auslander_isolated_1986} we have that $\Gamma \cong \Hom_S(Q,S)$ for some projective $\Gamma^{\textup{op}}$-module $Q$. The functor $\Hom_S(\square,S):\MCM(\Gamma) \to \MCM(\Gamma^{\textup{op}})$ defines a duality and therefore
\[Q \cong \Hom_S(\Hom_S(Q,S),S) \cong \Hom_S(_\Gamma\Gamma,S).\] Thus, $\Hom_S(_\Gamma\Gamma,S)$ is $\Gamma^{\textup{op}}$-projective which happens if and only if $\Hom_S(_{\Gamma^{\textup{op}}}\Gamma,S)$ is $\Gamma$-projective according to \cite[Lemma 5.1]{auslander_isolated_1986}. In this case, \cite[5.2]{auslander_isolated_1986} says that $\text{injdim}_\Gamma\Gamma = \text{injdim}_SS$ which is equal to $r$ since $S$ is Gorenstein. Interchanging the roles of $\Gamma$ and $\Gamma^{\textup{op}}$ we find that $\text{injdim}_{\Gamma^{\textup{op}}}\Gamma = r$ as well.
\end{proof}

\begin{remark}
If $\bf k$ is algebraically closed and its characteristic does not divide $d$, both Proposition \ref{thm:periodicity_in_MCM_Gamma} and Lemma \ref{thm:inj_dim_Gamma} can be restated for the skew group algebra over the $d$-fold branched cover $R^\sharp[\sigma]$ defined in Section \ref{section:d_branched_cover}.
\end{remark}

Over the hypersurface ring $R=S/(f)$, the reduced syzygy of an indecomposable non-free MCM $R$-module is again indecomposable. Proposition \ref{thm:Herzog_part2} gives an analogous result for matrix factorizations. First, we describe the relationship between projective summands in $\MatFac{S}{d}{f}$ and free summands in $\MCM(R)$. Recall that a matrix factorization $X \in \MatFac{S}{d}{f}$ is \textit{stable} if $\cok\phi_k$ is a stable MCM $R$-module for all $k \in \Z_d$.

\begin{prop}\label{thm:proj_summands}
Let $X = (\phi_1,\dots,\phi_d)\in \MatFac{S}{d}{f}$ and set $M_i = \cok\phi_i$ for all $i \in \Z_d$. Then, $X$ has a projective summand isomorphic to $\mc P_i$ if and only if $M_i$ has a free $R$-summand.
\end{prop}

\begin{proof} The statement holds when $d=2$ (for instance see \cite[7.5]{yoshino_cohen-macaulay_1990}). So, assume $d\ge 3$. One direction is immediate: If $X \cong X' \oplus \mc P_i$ for some $X'=(\phi_1',\dots,\phi_d')\in \MatFac{S}{d}{f}$ and $i \in \Z_d$, then $M_i \cong \cok\phi_i' \oplus R$. 

A matrix factorization $Y$ is a summand of $X$ if and only if $T^j(Y)$ is a summand of $T^j(X)$ for any $j \in \Z_d$. Therefore, for the converse, we may assume $i=1$. That is, assume $M_1 \cong M \oplus R$ for some MCM $R$-module $M$. By Proposition \ref{thm:Eisenbud_d=2} \ref{thm:Eisenbud_d=2:existence}, there exists $(\phi:G\to F,\psi:F \to G) \in\MatFac{S}{2}{f}$, with $\phi$ minimal, such that $\cok\phi \cong M$. Then, 
\[\begin{tikzcd}[ampersand replacement=\&]
0 \rar \&G\oplus S \rar{\begin{pmatrix}\phi & 0\\ 0 &f\end{pmatrix}}\&F\oplus S \rar \&M_1 \rar \&0
\end{tikzcd}\] is a minimal free resolution of $M_1$ over $S$. Thus, there exists isomorphisms $\alpha$ and $\beta$ and a commutative diagram
\[\begin{tikzcd}[ampersand replacement=\&,column sep=1cm]
0 \rar \&F_2 \dar{\beta} \rar{\phi_1} \&F_1\dar{\alpha} \rar \&M_1 \dar[equals] \rar \&0\\
0 \rar \&G\oplus S \oplus S^m \rar[swap]{\scalebox{.8}{\(\begin{pmatrix}\phi & 0 &0\\ 0 &f &0\\ 0 &0 &1_{S_m}\end{pmatrix}\)}}\&F\oplus S \oplus S^m \rar \&M_1 \rar \&0
\end{tikzcd}\] for some $m\ge 0$. In other words, we have an isomorphism of matrix factorizations in $\MatFac{S}{2}{f}$:
\[(\phi_1,\phi_2\phi_3\cdots\phi_d)\cong\left(\scalebox{.8}{\(\begin{pmatrix}\phi & 0 &0\\ 0 &f &0\\ 0 &0 &1_{S_m}\end{pmatrix},\begin{pmatrix}\psi & 0 &0\\ 0 &1 &0\\ 0 &0 &f\cdot 1_{S_m}\end{pmatrix}\)}\right).\] The isomorphisms $\alpha$ and $\beta$ also give us an isomorphism of matrix factorizations in $\MatFac{S}{d}{f}$:
\[\begin{split}X &\cong (\alpha\phi_1,\phi_2,\dots,\phi_{d-1},\phi_d\alpha^{-1})\\
                 &\cong (\alpha\phi_1\beta^{-1},\beta\phi_2,\dots\phi_{d-1},\phi_d\alpha^{-1}).
\end{split}\] Let $p_1: F\oplus S \oplus S^m \to S$ and $p_2: G\oplus S \oplus S^m \to S$ be projection onto the middle components of $F\oplus S \oplus S^m$ and $G\oplus S\oplus S^m$ respectively. Consider the diagram
\[\scalebox{.9}{\begin{tikzcd}[ampersand replacement = \&, row sep=3em]
F\oplus S \oplus S^m \dar{p_1} \rar{\phi_d\alpha^{-1}} \&F_d \dar{p_2\beta\phi_2\phi_3\cdots\phi_{d-1}} \rar{\phi_{d-1}} \&\cdots \rar{\phi_4} \&F_4 \dar{p_2\beta\phi_2\phi_3} \rar{\phi_3} \&F_3 \dar{p_2\beta\phi_2} \rar{\beta\phi_2} \&G\oplus S \oplus S^m \dar{p_2} \rar{\alpha\phi_1\beta^{-1}}\&F\oplus S\oplus S^m \dar{p_1}\\
S \rar{1} \&S \rar{1} \&\cdots \rar{1} \&S \rar{1} \&S \rar{1} \&S \rar{f} \&S.
\end{tikzcd}}\] The two right most squares commute, the first since $\alpha\phi_1\beta^{-1} = \scalebox{.7}{\(\begin{pmatrix}\phi & 0 &0\\ 0 &f &0\\ 0 &0 &1_{S_m}\end{pmatrix}\)}$ and the second by construction. Similarly, for $k=3,4,\dots,d-1$, the square
\[\begin{tikzcd}
F_{k+1} \dar[swap]{p_2\beta\phi_2\phi_3\cdots\phi_{k-1}\phi_k} \rar{\phi_k} &F_k \dar{p_2\beta\phi_2\phi_3\cdots\phi_{k-1}}\\
S \rar{1} &S
\end{tikzcd}\] also commutes. Since $p_1\alpha\phi_1\beta^{-1} = fp_2$, we have that
\[\begin{split}
    f p_1 = p_1 f &= p_1\alpha\phi_1\beta^{-1}\beta\phi_2\phi_3\cdots\phi_{d-1}\phi_d\alpha^{-1}\\
                  &= f p_2  \beta\phi_2\phi_3\cdots\phi_{d-1}\phi_d\alpha^{-1}.
\end{split}\] We may cancel $f$ on the left to conclude that the left most square commutes as well. Thus, we have a morphism
\[X\cong (\alpha\phi_1\beta^{-1},\beta\phi_2,\dots,\phi_{d-1},\phi_d\alpha^{-1}) \to \mc P_1.\] We claim that this morphism is an admissible epimorphism. By Lemma \ref{thm:admiss_morphs_form}, it suffices to show that each of the vertical maps depicted above are surjective. By the commutativity of the diagram,
\[ p_1 =  (p_2\beta\phi_2\phi_3\cdots\phi_k)(\phi_{k+1}\phi_{k+2}\cdots\phi_{d-1}\phi_d\alpha^{-1})\] for each $k=2,3,\dots,d-1$. Since $p_1$ is surjective, this implies $p_2\beta\phi_2\phi_3\cdots\phi_k$ is surjective for each $k=2,3,\dots,d-1$ as claimed. Since $\mc P_1$ is projective, the admissible epimorphism $X \twoheadrightarrow P_1$ implies that $X$ has a direct summand isomorphic to $\mc P_1$.
\end{proof}

We note that KRS was not needed for the proof of Proposition \ref{thm:proj_summands}. The result therefore holds for an arbitrary regular local ring $S$. The following corollary also holds without the completeness of $S$.

\begin{cor}\label{thm:only_indecomp_projs}
Let $S$ be a regular local ring, $f$ a non-zero non-unit in $S$, and $d \ge 2$.
\begin{enumerate}[label = (\roman*)]
    \item \label{thm:only_indecomp_projs:1} The objects $\mc P_1,\mc P_2,\dots,\mc P_d \in\MatFac{S}{d}{f}$ are the only indecomposable projectives (equivalently injectives) up to isomorphism.
    \item \label{thm:only_indecomp_projs:2} A matrix factorization $X \in \MatFac{S}{d}{f}$ is stable if and only if it has no non-zero projective direct summands.
\end{enumerate}

\end{cor}

\begin{proof} Let $Q = (Q_1,Q_2,\dots,Q_d) \in\MatFac{S}{d}{f}$ be an indecomposable projective of size $n$. Then, we have a short exact sequence of the form \eqref{equation:enough_proj}, and more specifically, $Q$ is a direct summand of $P(Q) \cong \bigoplus_{i \in \Z_d}\mc P_i^n$. It follows that, for any $k \in \Z_d$, $\cok Q_k$ is either 0 or has a direct summand isomorphic to $R$. Since $Q$ is non-zero matrix factorization, there exists $j \in \Z_d$ such that $\cok Q_j \neq 0$. Thus, $\cok Q_j$ has a direct summand isomorphic to $R$. By Proposition \ref{thm:proj_summands}, this implies that $Q$ has a direct summand isomorphic to $\mc P_j$. Since $Q$ is indecomposable, we have that $Q \cong \mc P_j$. The statement about indecomposable injectives follows immediately because of Lemma \ref{thm:proj_iff_inj}.

The second statement follows by combining \ref{thm:only_indecomp_projs:1} and Proposition \ref{thm:proj_summands}.
\end{proof}

Proposition \ref{thm:proj_summands} and Corollary \ref{thm:only_indecomp_projs} give us a clearer picture of the structure of matrix factorizations and the MCM $R$-modules they encode.

\begin{prop}\label{thm:stable+proj}
Let $X = (\phi_1,\dots,\phi_d) \in \MatFac{S}{d}{f}$. Then,
\begin{equation}\label{equation:stable+proj} X \cong \tilde X \oplus \mc P_1^{s_1}\oplus \mc P_2^{s_2}\oplus\cdots\oplus \mc P_d^{s_d}\end{equation} for some stable matrix factorization $\tilde X = (\tilde\phi_1,\dots,\tilde\phi_d)$ and integers $s_k \ge 0$, $k \in \Z_d$. The integers $s_k$ are uniquely determined and the direct summand $\tilde X$ is unique up to isomorphism. The corresponding MCM $R$-modules are
\[\cok\phi_k \cong \cok\tilde\phi_k \oplus R^{s_k}\] where $\cok\tilde\phi_k$ is a stable MCM $R$-module for all $k \in \Z_d$.
\end{prop}

\begin{proof}
Lemma \ref{thm:only_indecomp_projs} says that the only indecomposable projective matrix factorizations are $\mc P_1,\mc P_2,\dots,\mc P_d$. Since $S$ is complete, KRS implies that we have a decomposition $X \cong \tilde X \oplus \mc P_1^{s_1}\oplus \mc P_2^{s_2}\oplus\cdots\oplus \mc P_d^{s_d}$ for some unique integers $s_k\ge 0$ and some $\tilde X$ which has no projective summands. Finally, using Proposition \ref{thm:proj_summands}, each copy of $\mc P_k$ contributes a free $R$-summand to $\cok\phi_k$.
\end{proof}

\begin{cor}\label{thm:omega_stable+proj}
Let $X=(\phi_1,\dots,\phi_d) \in \MatFac{S}{d}{f}$ of size $n$. Then,
\begin{equation}\label{equation:omega_stable+proj}
    \Omega_{\MatFac{S}{d}{f}}(X) \cong \tilde\Omega \oplus \mc P_1^{m_1}\oplus \cdots\oplus \mc P_d^{m_d}
\end{equation} where $m_k = n- \mu_R(\cok\phi_k)$ and $\tilde\Omega \in \MatFac{S}{d}{f}$ is stable. Furthermore, $\tilde\Omega$ is of size $\sum_{k=1}^d\mu_R(\cok\phi_k) - n$.
\end{cor}

\begin{proof}
The isomorphism \eqref{equation:omega_stable+proj} follows by combining Propositions \ref{thm:stable+proj} and \ref{thm:cok_Omega_k}. Let $\ell \ge 0$ be the size of $\tilde\Omega$. Since $\Omega_{\MatFac{S}{d}{f}}(X)$ is of size $(d-1)n$, we have that
\[(d-1)n = \ell + \sum_{k=1}^d m_k = \ell + dn - \sum_{k=1}^d\mu_R(\cok\phi_k).\] Thus, $\ell = \sum_{k=1}^d\mu_R(\cok\phi_k) - n$.
\end{proof}

We continue with two more important properties of the stable part of the syzygy of a matrix factorization.

\begin{lem}\label{thm:Herzog_part1}
Let $X \in \MatFac{S}{d}{f}$ and suppose 
\[\begin{tikzcd}
K \rar[tail] &P \rar[two heads] &X
\end{tikzcd}\] is a short exact sequence with $P$ projective. Then, $\tilde\Omega$ is isomorphic to a direct summand of $K$ where $\Omega_{\MatFac{S}{d}{f}}(X)\cong\tilde\Omega\oplus Q$ is a decomposition of the form \eqref{equation:omega_stable+proj}.
\end{lem}

\begin{proof}
This follows directly from Lemma \ref{thm:Schanuel} and Proposition \ref{thm:stable+proj}.
\end{proof}

\begin{prop}\label{thm:Herzog_part2}
Let $X \in \MatFac{S}{d}{f}$ be a indecomposable non-projective matrix factorization and let $\Omega_{\MatFac{S}{d}{f}}(X) 
\cong\tilde\Omega \oplus P$ be a decomposition of the form \eqref{equation:omega_stable+proj}. Then, $\tilde\Omega$ is indecomposable.
\end{prop}

\begin{proof}
First, note that, $\tilde\Omega \neq 0$. Indeed, if $\Omega_{\MatFac{S}{d}{f}}(X)$ was projective, then Proposition \ref{thm:omega_additive} would imply that $X$ is projective as well, which is not the case. So, assume $\tilde\Omega = Y_1 \oplus Y_2$ for some non-zero $Y_1,Y_2 \in \MatFac{S}{d}{f}$. Since $\tilde\Omega$ is stable, the direct summands $Y_1$ and $Y_2$ are also stable. Then,
\[\begin{split}
    \Omega_{\MatFac{S}{d}{f}}^{-}(\Omega_{\MatFac{S}{d}{f}}(X)) &\cong  \Omega_{\MatFac{S}{d}{f}}^{-}(\tilde\Omega) \oplus  \Omega_{\MatFac{S}{d}{f}}^{-}(P)\\
    &\cong  \Omega_{\MatFac{S}{d}{f}}^{-}(Y_1) \oplus  \Omega_{\MatFac{S}{d}{f}}^{-}(Y_2) \oplus  \Omega_{\MatFac{S}{d}{f}}^{-}(P).
\end{split}\] For $i=1,2$, decompose $ \Omega_{\MatFac{S}{d}{f}}^{-}(Y_i) \cong U_i \oplus P_i$, for some stable $U_i$ and projective $P_i$. Now, applying Proposition \ref{thm:omega_iso_omega-1}, we have that
\[ X \oplus \mc P^{(d-2)n} \cong U_1 \oplus U_2 \oplus P_1\oplus P_2 \oplus  \Omega_{\MatFac{S}{d}{f}}^{-}(P) \] where $n$ is the size of $X$. Since both sides of this isomorphism are decomposed into the form \eqref{equation:stable+proj}, we have that $U_1\oplus U_2 \cong X$. But $X$ is indecomposable, so one of $U_1$ or $U_2$ must be zero. Re-indexing if necessary, we may assume $U_1=0$. This implies that $\Omega_{\MatFac{S}{d}{f}}^{-}(Y_1)$ is projective and therefore $Y_1$ is projective. However, this is a contradiction since $Y_1$ is a non-zero stable matrix factorization. Hence, $\tilde\Omega$ is indecomposable.
\end{proof}

So far, we have refrained from assuming that $X$ is a reduced matrix factorization. On the other hand, if we do assume that $X \in \MatFac{S}{d}{f}$ is reduced, we obtain more concise versions of \ref{thm:cok_Omega_k}, \ref{thm:stable+proj}, \ref{thm:omega_stable+proj}, and \ref{thm:Herzog_part2}.

\begin{cor}\label{thm:prop_of_reduced_MFs}
Let $X \in \MatFac{S}{d}{f}$ be reduced. Then, the following hold.

\begin{enumerate}[label = (\roman*)]
    \item $\cok\Omega_k \cong \syz_R^1(\cok\phi_k)$ for each $k \in \Z_d$.
    \item Both $X$ and $\Omega_{\MatFac{S}{d}{f}}(X)$ are stable.
    \item If $X$ is indecomposable, then $\Omega_{\MatFac{S}{d}{f}}(X)$ is indecomposable.
\end{enumerate}
\end{cor}
\begin{proof}
By Proposition \ref{thm:Eisenbud_d=2}\ref{thm:Eisenbud_d=2:reduced}, there is a one-to-one correspondence between reduced matrix factorizations in $\MatFac{S}{2}{f}$ and stable MCM $R$-modules (Lemma \ref{thm:Eisenbud_d=2}). If $X=(\phi_1,\phi_2,\dots,\phi_d) \in \MatFac{S}{d}{f}$ is reduced, then $(\phi_k,\phi_{k+1}\phi_{k+2}\cdots\phi_{k-1})$ is a reduced matrix factorization in $\MatFac{S}{2}{f}$ for each $k \in \Z_d$. Hence, $\cok\phi_k$ is a stable MCM $R$-module for each $k \in \Z_d$ and $\cok(\phi_{k+1}\phi_{k+2}\cdots\phi_{k-1})$ is its reduced first syzygy. Since, by Proposition \ref{thm:cok_Omega_k}, $\cok\Omega_k \cong \cok(\phi_{k+1}\phi_{k+2}\cdots\phi_{k-1})$, the first statement follows. The second statement follows from Proposition \ref{thm:proj_summands} and the third follows from the second and Proposition \ref{thm:Herzog_part2}.
\end{proof}


\section{Examples}\label{section:examples}

Let $S$ be a regular local ring with maximal ideal $\mf n$ and let $f$ be a non-zero non-unit in $S$. An important distinction between the categories $\MatFac{S}{2}{f}$ and $\MatFac{S}{d}{f}$, for $d >2$, is the existence of stable non-reduced matrix factorizations. We illustrate this with our first example.

\begin{example}\label{example:f1_fd}
Let $d>2$ and assume $f \in S$ can be written as $f=f_1f_2\cdots f_d$ for some $f_i \in \mf n$. Then, $(f_1,f_2,\dots,f_d)$ is an indecomposable reduced matrix factorization of size 1 in $\MatFac{S}{d}{f}$. Since $d>2$, the syzygy of $(f_1,f_2,\dots,f_d)$ in $\MatFac{S}{d}{f}$ is not reduced. However, Corollary \ref{thm:prop_of_reduced_MFs} implies that $\Omega(f_1,f_2,\dots,f_d)$ is indecomposable and stable.

It is tempting to think that the presence of unit entries must make $\Omega(f_1,\dots,f_d)$ redundant in some way. Actually, we can not replace $\Omega(f_1,\dots,f_d)$ with a reduced matrix factorization in the following sense.

Set $\Omega(f_1,\dots,f_d) = (\Omega_1,\Omega_2,\dots,\Omega_d)$ and let $\widehat f_k = f_1\cdots f_{k-1}f_{k+1}\cdots f_d$ for each $k \in \Z_d$. Then, by Proposition \ref{thm:cok_Omega_k}, we have that $\cok\Omega_k \cong R/(\widehat f_k)$. Suppose $X = (\phi_1,\dots,\phi_d)$ is a reduced matrix factorization of $f$ such that $\cok\phi_k \cong R/(\widehat f_k)$ for all $k \in \Z_d$. Since $\phi_k$ is minimal, the short exact sequence
\[\begin{tikzcd}
0 \rar &F_{k+1} \rar{\phi_k} &F_k \rar &R/(\widehat f_k) \rar &0
\end{tikzcd}\] is a minimal free resolution of $R/(\widehat f_k)$ over $S$. By tensoring with $R$, it follows that $\rank_S F_k = \mu_R(R/(\widehat f_k)) = 1$. That is, $X$ is a matrix factorization of size $1$. Thus, the homomorphism $\phi_k$ is given by multiplication by $\widehat f_k$ up to a unit, say $\phi_k = v_k\widehat f_k$ for $v_k \in S$ invertible. 
Since $X$ is a matrix factorization of $f$, we have that
\[f = \phi_1\phi_2\cdots\phi_d = v\widehat f_1\widehat f_2\cdots \widehat f_d = v f^{d-1}\] where $v = v_1v_2\cdots v_d$. Since $d> 2$, this is a contradiction. We conclude that no reduced matrix factorization has the same cokernels as $\Omega(f_1,\dots,f_d)$.
\end{example}




With Example \ref{example:f1_fd} in mind, we continue to focus on stable matrix factorizations. However, we introduce an additional assumption (which implies stability) in order to avoid particularly trivial factorizations.

\begin{defi}
Let $S$ be a regular local ring, $f$ a non-zero non-unit in $S$, and $d \ge 2$. A non-zero matrix factorization $X=(\phi_1,\phi_2,\dots,\phi_d) \in \MatFac{S}{d}{f}$ is \textit{pseudoprojective} if $\cok\phi_k = 0$ for some $k \in \Z_d$.
\end{defi}

We aim to avoid matrix factorizations with pseudoprojective summands. Notice that any indecomposable projective (equivalently injective) in $\MatFac{S}{d}{f}$ is pseudoprojective (See \ref{thm:only_indecomp_projs}). 

An indecomposable matrix factorization in $\MatFac{S}{2}{f}$ is pseudoprojective if and only if it is projective; the only possibilities are $(1,f)$ and $(f,1)$. Consequently, $X \in \MatFac{S}{2}{f}$ is reduced if and only if it is stable if and only if it has no pseudoprojective summands (compare with \ref{thm:Eisenbud_d=2} \ref{thm:Eisenbud_d=2:reduced}). However, for $d > 2$, these properties are not equivalent. The following lemma and examples give the precise relationship.

\begin{lem}\label{thm:reduced_pseudoproj_stable}
Let $X \in \MatFac{S}{d}{f}$ with $d\ge 2$.
\begin{enumerate}[label = (\roman*)]
    \item \label{thm:reduced_pseudoproj_stable:1} If $X$ is reduced, then $X$ has no pseudoprojective summands.
    \item \label{thm:reduced_pseudoproj_stable:2} If $X$ has no pseudoprojective summands, then $X$ is stable.
\end{enumerate}
\end{lem}

\begin{proof}
In order to prove \ref{thm:reduced_pseudoproj_stable:1}, suppose $X \cong X' \oplus X''$ for some $X'=(\phi_1',\dots,\phi_d')$ and $X''=(\phi_1'',\dots,\phi_d'')$ with $X'$ pseudoprojective. Then, $\cok\phi_k' = 0$ for some $k \in \Z_d$. Equivalently, $\phi_k'$ is an isomorphism and therefore, $\phi_k \cong \phi_k' \oplus \phi_k''$ is not minimal.

Corollary \ref{thm:only_indecomp_projs}\ref{thm:only_indecomp_projs:2} showed that $X$ is stable if and only if it has no non-zero projective summands. Since indecomposable projective matrix factorizations are pseudoprojective, \ref{thm:reduced_pseudoproj_stable:2} follows.
\end{proof}

Examples \ref{example:pseudoprojs} and \ref{example:D_infty} below show that the converses of \ref{thm:reduced_pseudoproj_stable:2} and \ref{thm:reduced_pseudoproj_stable:1} respectively fail in general for $d > 2$.

\begin{example}\label{example:pseudoprojs}
Let $(\phi:G \to F,\psi:F \to G) \in \MatFac{S}{2}{f}$ be a reduced matrix factorization.
\begin{enumerate}[label = (\roman*)]
    \item $(\phi,\psi,1_F) \in\MatFac{S}{3}{f}$ is stable but is itself pseudoprojective.
    \item $\left(\begin{pmatrix}
    \phi &\\
    & 1_F
    \end{pmatrix},\begin{pmatrix}
    \psi &\\
    & 1_F
    \end{pmatrix}, \begin{pmatrix}
    1_F &\\
    & \phi
    \end{pmatrix},\begin{pmatrix}
    1_F &\\
    & \psi
    \end{pmatrix}\right) \in\MatFac{S}{4}{f}$ is stable but it is a direct sum of pseudoprojectives.
\end{enumerate}
\end{example}

\begin{example}\label{example:D_infty}
Let $S=\mathbf{k} \llbracket x,y\rrbracket$, with $\bf k$ an uncountable algebraically closed field of characteristic different from 2. Let $f = x^2y \in S$ and, as usual, set $R=S/(f)$. Consider the following matrix factorization of $f$ with 3 factors:
\[X = (\phi_1,\phi_2,\phi_3) = \left(
\begin{pmatrix}
x &y\\
0 &-x
\end{pmatrix},
\begin{pmatrix}
0 &y\\
x^2 &-x
\end{pmatrix},
\begin{pmatrix}
1 &0\\
x &y
\end{pmatrix}\right).\] Notice that $X$ is not reduced. We claim that $X$ is indecomposable. To see this, first note that the $R$-module $\cok\begin{pmatrix}
x &y\\
0 &-x
\end{pmatrix}$ is an indecomposable MCM over the $D_\infty$ hypersurface singularity $R= \mathbf{k}\llbracket x ,y \rrbracket/(x^2y)$ \cite[Proposition 4.2]{buchweitz_cohen-macaulay_1987}. If $X \cong Y \oplus Y'$ for some non-zero matrix factorizations $Y=(u,v,w),Y'=(u',v',w') \in \MatFac{S}{3}{f}$, both necessarily of size 1, then 
\[\cok\begin{pmatrix}
x &y\\
0 &-x
\end{pmatrix} \cong \cok(u) \oplus \cok(u').\] This implies that one of $\cok(u)$ or $\cok(u')$ must be zero. Rearranging the summands if necessary, we may assume $\cok(u') = 0$. Equivalently, we have that $u'$ is multiplication by a unit in $S$. This would imply a matrix equivalence
\[\begin{pmatrix}
x &y\\
0 &-x
\end{pmatrix} \sim \begin{pmatrix}
u &0\\
0 &u'
\end{pmatrix}\] which is not possible since the matrix on the right has a unit entry. Therefore, no such decomposition is possible. Since $X$ is indecomposable and the matrices $\phi_1,\phi_2,$ and $\phi_3$ are non-invertible, it follows that $X$ has no pseudoprojective summands.

Set $\mu_i = \mu_R(\cok\phi_i)$, $i\in \Z_3$. Then, $\mu_1 = 2, \mu_2=2,$ and $\mu_3 = 1$. Corollary \ref{thm:omega_stable+proj} says that we should expect a direct sum decomposition of $\Omega_{\MatFac{S}{3}{f}}(X)$ into a stable matrix factorization of size $3$ and a projective isomorphic to $\mc P_3 = (1,1,x^2y)$. Indeed, using invertible row and column operations, one can construct an isomorphism
\[\begin{split}\Omega_{\MatFac{S}{3}{f}}(X) &=
\left(\scalebox{.8}{\(\begin{pmatrix}
0 &-y &-xy &-y^2\\
-x^2 &x &0 &xy\\
1 &0 &0 &0\\
0 &1 &0 &0
\end{pmatrix},
\begin{pmatrix}
-1 &0 &-x &-y\\
-x &-y &-x^2 &0\\
1 &0 &0 &0\\
0 &1 &0 &0
\end{pmatrix},
\begin{pmatrix}
-x &-y &-x^2y &0\\
0 &x &x^3 &-x^2\\
1 &0 &0 &0\\
0 &1 &0 &0
\end{pmatrix}\)}\right)\\
&\cong 
\left(\scalebox{.8}{\(\begin{pmatrix}
1 &0 &0 &0\\
0 &x &0 &xy\\
0 &-y & -xy &-y^2\\
0 &1 &0 &0
\end{pmatrix},
\begin{pmatrix}
1 &0 &0 &0\\
0 &-y &0 &xy\\
0 &0 & -x &-y\\
0 &1 &0 &0
\end{pmatrix},
\begin{pmatrix}
x^2y &0 &0 &0\\
0 &x &0 &-x^2\\
0 &0 &1 &0\\
0 &1 &0 &0
\end{pmatrix}\)}\right).\end{split}\]
 Proposition \ref{thm:Herzog_part2} implies that the stable part of $\Omega_{\MatFac{S}{3}{f}}(X)$,
\[\tilde\Omega = \left(\begin{pmatrix}
x &0 &xy\\
-y &-xy &-y^2\\
1 &0 &0
\end{pmatrix},\begin{pmatrix}
-y &0 &xy\\
0 &-x &-y\\
1 &0 &0
\end{pmatrix},\begin{pmatrix}
x &0 &-x^2\\
0 &1 &0\\
1 &0 &0
\end{pmatrix}\right),\] is also an indecomposable matrix factorization of $f=x^2y$. Clearly, none of the matrices that make up $\tilde\Omega$ are invertible and therefore $\tilde\Omega$ has no pseudoprojective summands.
\end{example}

One benefit of having no pseudoprojective summands is the following.

\begin{lem}\label{thm:pseudoproj_indecomp}
Assume $X = (\phi_1,\phi_2,\dots,\phi_d) \in \MatFac{S}{d}{f}$ has no pseudoprojective summands. If $\cok\phi_j$ is an indecomposable MCM $R$-module for some $j \in \Z_d$, then $X$ is an indecomposable matrix factorization.
\end{lem}

\begin{proof}
Suppose $X \cong X' \oplus X''$ for non-zero $X' = (\phi_1',\dots,\phi_d'),X'' = (\phi_1'',\dots,\phi_d'') \in \MatFac{S}{d}{f}$. Then, $\cok\phi_j \cong \cok\phi_j' \oplus \cok\phi_j''$. Since $\cok\phi_j$ is indecomposable, it follows that one of $\phi_j'$ or $\phi_j''$ is an isomorphism. That is, one of $X'$ or $X''$ is pseudoprojective, a contradiction. Hence, $X$ is indecomposable.
\end{proof}

Any partial converse to Lemma \ref{thm:pseudoproj_indecomp} would be very interesting but would likely need more conditions on $X$. The only result we have in this direction is a special case of Proposition \ref{thm:proj_summands} which says that the cokernels of an indecomposable non-projective matrix factorization have no free summands.

In the case that $f$ is an irreducible in $S$, we have another interesting consequence.

\begin{prop}
Assume $f \in S$ is irreducible. Then, any $X \in \MatFac{S}{d}{f}$ which is not pseudoprojective must be of size at least $d$.
\end{prop}

\begin{proof}
Let $X = (\phi_1:F_2 \to F_1,\dots,\phi_d:F_1 \to F_d)$ and set $n$ to be the size of $X$. Since $\phi_1\phi_2\cdots\phi_d = f\cdot 1_{F_1}$, we have that
\[(\det\phi_1)(\det \phi_2)\cdots (\det \phi_d) = f^n.\] Since $f$ is irreducible, for each $k \in \Z_d$, there exists units $u_k \in S$ and integers $r_k \ge 0$ such that $\det\phi_k = u_k f^{r_k}$. Moreover, we have that $n = \sum_{k \in \Z_d}r_k$.

Let $k \in \Z_d$. Since $(\phi_k,\phi_{k+1}\phi_{k+2}\cdots\phi_{k-1}) \in \MatFac{S}{2}{f}$, \cite[Proposition 5.6]{eisenbud_homological_1980} implies that $\rank_R(\cok\phi_k) = r_k$, where $R=S/(f)$ as usual. The non-negative integer $r_k = 0$ if and only if $\phi_k$ is an isomorphism. However, since $X$ is not pseudoprojective, we have that $r_k > 0$. Thus, $n  = \sum_{k \in \Z_d}r_k \ge d$.
\end{proof}

In the case $d=2$, this proposition is simply a restatement of the definition of an irreducible element; if $(u,v) \in \MatFac{S}{2}{f}$ is of size 1, then one of $u$ or $v$ must be a unit in $S$.

\begin{example}
Let $\bf k$ be an algebraically closed field of characteristic 0 and set $S = \mathbf{k}\llbracket x ,y \rrbracket$. Let $\omega \in \mathbf{k}$ be a primitive $3$rd root of unity. Consider the simple curve singularity $R = S/(x^3 + y^4)$ of type $\mathbb{E}_6$. Following the notation of \cite[Chapter 9]{yoshino_cohen-macaulay_1990}, the indecomposable MCM $R$-module $B$ is given by the matrix factorization
\[(\beta, \alpha) = \left( \begin{pmatrix}
    y &0 &x\\
    x &-y^2 &0\\
    0 &x &-y
    \end{pmatrix}, \begin{pmatrix}
    y^3 &x^2 &xy^2\\
    xy &-y^2 &x^2\\
    x^2 &-xy &-y^3
    \end{pmatrix}\right) \in \MatFac{S}{2}{x^3+y^4}\] in the sense that $B = \cok\beta$. The matrix $\alpha$ factors further which gives rise to a matrix factorization of $x^3 + y^4$ with 3 factors:
    \[(\beta,\beta_2,\beta_3) \coloneqq \left(
    \begin{pmatrix}
    y &0 &x\\
    x &-y^2 &0\\
    0 &x &-y
    \end{pmatrix},
    \begin{pmatrix}
    -y^2 &0 &\omega x\\
    \omega x &-y &0\\
    0 &\omega x &y
    \end{pmatrix},
    \begin{pmatrix}
    -y &0 &\omega^2 x\\
    \omega^2 x &y &0\\
    0 &\omega^2 x &-y^2
    \end{pmatrix} 
    \right)\]

As we can see, the entries of $\beta,\beta_2,$ and $\beta_3$ lie in the maximal ideal of $S$, that is, $(\beta,\beta_2,\beta_3) \in \MatFac{S}{3}{x^3+y^4}$ is reduced. Combining Lemma \ref{thm:reduced_pseudoproj_stable}, Lemma \ref{thm:pseudoproj_indecomp}, and the fact that $B = \cok\beta$ is an indecomposable MCM $R$-module, we conclude that $(\beta,\beta_2,\beta_3)$ is indecomposable. We also note that, since $\det \beta = x^3 + y^4$, the matrix $\alpha$ is the adjoint of $\beta$. Hence, $\alpha = \beta_2\beta_3$ is a factorization of the adjoint.

The matrix factorization $(\beta,\beta_2,\beta_3)$ is essentially the one constructed in \cite[Proposition 2.1]{blaser_ulrich_2017}. Similar examples are also found using this construction. For instance,

\begin{enumerate}
    \item[$(\mathbb{E}_7)$] Let $f = x^3 + xy^3 \in S$.
    \[\left(
    \begin{pmatrix}
    y &0 &x\\
    -x &xy &0\\
    0 &-x &y
    \end{pmatrix},
    \begin{pmatrix}
    xy &0 &\omega x\\
    -\omega x &y &0\\
    0 &-\omega x &y
    \end{pmatrix},
    \begin{pmatrix}
    y &0 & \omega^2 x\\
    -\omega^2 x &y &0\\
    0 &-\omega^2 x &xy
    \end{pmatrix}\right) \in \MatFac{S}{3}{x^3+xy^3}\]
    
    \item[$(\mathbb{E}_8)$] Let $f =x^3 +y^5 \in S$.
    \[\left(
    \begin{pmatrix}
    y &-x &0\\
    0 &y &-x\\
    x &0 &y^3
    \end{pmatrix},
    \begin{pmatrix}
    y^3 &-\omega x &0\\
    0 &y &-\omega x\\
    \omega x &0 &y
    \end{pmatrix},
    \begin{pmatrix}
    y &-\omega^2 x &0\\
    0 &y^3 &-\omega^2 x\\
    \omega^2 x &0 &y
    \end{pmatrix}\right) \in \MatFac{S}{3}{x^3+y^5}\] and
    \[\left(
    \begin{pmatrix}
    y &-x &0\\
    0 &y^2 &-x\\
    x &0 &y^2
    \end{pmatrix},
    \begin{pmatrix}
    y^2 &-\omega x &0\\
    0 &y &-\omega x\\
    \omega x &0 &y^2
    \end{pmatrix},
    \begin{pmatrix}
    y^2 &-\omega^2 x &0\\
    0 &y^2 &-\omega^2 x\\
    \omega^2 x &0 &y
    \end{pmatrix}\right) \in \MatFac{S}{3}{x^3+y^5}.\]    
\end{enumerate}
\end{example}

\section*{Acknowledgements}
The author would like to thank his advisor, Graham J. Leuschke, for his dedication to this project and for his guidance and encouragement throughout its preparation.

\bibliographystyle{amsalpha}
\bibliography{BibTex2020}

\email{\textit{E-mail address:} ttribone@syr.edu}

\address{Department of Mathematics, Syracuse University, Syracuse, New York 13244, USA}

\end{document}